\documentclass[12pt]{amsart}

\usepackage{psfrag}
\usepackage{color}

\usepackage{graphicx,graphics}
\usepackage{fullpage,amssymb,amsfonts,amsmath,amstext,amsthm,amscd,enumerate,verbatim,tikz,booktabs,subcaption,orcidlink}
\usepackage[T1]{fontenc}
\usetikzlibrary{matrix,arrows}

\begin{document}
\date{\today}

\newtheorem{theorem}{Theorem}[section]
\newtheorem{result}[theorem]{Result}
\newtheorem{fact}[theorem]{Fact}
\newtheorem{conjecture}[theorem]{Conjecture}
\newtheorem{definition}[theorem]{Definition}
\newtheorem{lemma}[theorem]{Lemma}
\newtheorem{proposition}[theorem]{Proposition}
\newtheorem{remark}[theorem]{Remark}
\newtheorem{corollary}[theorem]{Corollary}
\newtheorem{facts}[theorem]{Facts}
\newtheorem{question}[theorem]{Question}
\newtheorem{props}[theorem]{Properties}

\newtheorem{ex}[theorem]{Example}

\newcommand{\notes} {\noindent \textbf{Notes.  }}
\newcommand{\note} {\noindent \textbf{Note.  }}
\newcommand{\defn} {\noindent \textbf{Definition.  }}
\newcommand{\defns} {\noindent \textbf{Definitions.  }}
\newcommand{\x}{{\bf x}}
\newcommand{\z}{{\bf z}}
\newcommand{\B}{{\bf b}}
\newcommand{\V}{{\bf v}}
\newcommand{\T}{\mathcal{T}}
\newcommand{\Z}{\mathbb{Z}}
\newcommand{\Hp}{\mathbb{H}}
\newcommand{\D}{\mathbb{D}}
\newcommand{\R}{\mathbb{R}}
\newcommand{\N}{\mathbb{N}}
\renewcommand{\B}{\mathbb{B}}
\newcommand{\C}{\mathbb{C}}
\newcommand{\dt}{{\mathrm{det }\;}}
 \newcommand{\adj}{{\mathrm{adj}\;}}
 \newcommand{\0}{{\bf O}}
 \newcommand{\av}{\arrowvert}
 \newcommand{\zbar}{\overline{z}}
 \newcommand{\htt}{\widetilde{h}}
\newcommand{\ty}{\mathcal{T}}
\renewcommand\Re{\operatorname{Re}}
\renewcommand\Im{\operatorname{Im}}
\newcommand{\diam}{\operatorname{diam}}
\newcommand{\sign}{\operatorname{sign}}
\newcommand{\ft}{\widetilde{f}}
\newcommand{\Rot}{\mathcal{R}}

\newcommand{\ds}{\displaystyle}
\numberwithin{equation}{section}

\renewcommand{\theenumi}{(\roman{enumi})}
\renewcommand{\labelenumi}{\theenumi}

\title{Equipotentials and folds for quasiregular Mandelbrot sets}

\author{Alastair N. Fletcher\,\orcidlink{0000-0003-1942-6928}}
\email{afletcher@niu.edu}
\address{Department of Mathematical Sciences, Northern Illinois University, DeKalb, IL 60115-2888, USA}

\maketitle

\begin{abstract}
It is well-known that Douady and Hubbard proved that the Mandelbrot set is connected by constructing a uniformizing holomorphic map for its complement from the B\"ottcher coordinates. This strategy has a clear analogue in the setting of quasiregular Mandelbrot sets. In this paper, we show that the natural generalization of the Douady-Hubbard uniformizing map is quasiconformal in a neighbourhood of infinity, as well as, in fact, outside a certain neighbourhood of the quasiregular Mandelbrot set. On the other hand, we also show that this uniformizing map can reverse orientation at some parameters, which forces it to have folds and, in particular, prevents it from being injective. This feature cannot happen in the holomorphic setting. The connectivity of the quasiregular Mandelbrot sets therefore cannot be established along the Douady-Hubbard route, and remains open.
\end{abstract}

\section{Introduction}

As part of their foundational work in modern complex dynamics, Douady and Hubbard \cite{DH82} showed that the Mandelbrot set $\mathcal{M}$ is connected by explicitly constructing a uniformizing coordinate $\Psi : \C \setminus \mathcal{M} \to \C \setminus \overline{\D}$. More precisely, if $\psi_c$ is the B\"ottcher coordinate which straightens $g_c(z) := z^2+c$, for $c\in \C$, to $g_0(z) = z^2$ in a neighbourhood of infinity, then for $c\notin \mathcal{M}$, the function $\Psi(c) = \psi_c(c)$ is, first, well-defined and, second, provides the uniformizing holomorphic coordinate.

A quasiregular generalization of the Mandelbrot set was introduced by the author and Goodman in \cite{FG10}. Here, and throughout, we will fix $K> 1$ and $\theta \in (-\pi/2 , \pi/2 ]$. We denote by $h := h_{K,\theta}$ the affine stretch by factor $K$ in the direction $e^{i\theta}$. For $c\in \C$, we denote by $f_{c}$ the map $f_{c}:= g_c \circ h_{K,\theta}$. Then $f_{c}$ is a degree two quasiregular map with constant complex dilatation $\mu = e^{2i\theta} \frac{K-1}{K+1}$ and thus provides the simplest family of quasiregular maps in the plane in which to study their dynamics. The case $K=1$ reduces to the setting of quadratic polynomials.

Writing $I(f_c)$ for the escaping set of $f_c$ and $BO(f_c) = \C \setminus I(f_c)$ for the bounded orbit set, the generalized Mandelbrot set $\mathcal{M}_{K,\theta}$ is defined via
\begin{equation}
\label{def:mandelbrot}
\mathcal{M}_{K,\theta} = \{ c\in \C : 0 \in BO(f_c) \} = \{ c\in \C : f_c^n(0) \not \to \infty \text{ as } n \to \infty \}.
\end{equation}
When $K=1$, this reduces to the usual Mandelbrot set $\mathcal{M}$. Following the well-trodden path in complex dynamics, for $n\geq 0$ we set
\begin{equation}
\label{eq:Pn}
P_n(c) = f_{c}^n(c) = f_c^{n+1}(0).
\end{equation}
When $K=1$, the $P_n$ are polynomials in $c$. Understanding the roots and critical points of these polynomials is crucial for understanding the dynamics of the quadratic family. However, when $K>1$, the $P_n$ are no longer polynomials in $c$. Instead they are two dimensional polynomials in two variables, but understanding their behaviour is still crucial for the dynamical analysis of the family $f_{c}$.

The goal in this paper is to follow the Douady-Hubbard machinery as far as we are able and to point out what we can and cannot do. Part of our pathway was previously laid out: B\"ottcher coordinates for the family $f_c$ were shown to exist by the author and Fryer in \cite{FF12}. This provides a neighbourhood of infinity $U_c$ and an asymptotically conformal quasiconformal map $\psi_c:U_c \to \C$ such that $f_0 \circ \psi_c = \psi_c \circ f_c$ holds in $U_c$. Then \cite[Theorem 2.4(ii)]{FF12} shows that, when $c\notin \mathcal{M}_{K,\theta}$, the map $\psi_c$ may be extended injectively to a domain containing $z=c$. The extension is obtained by lifting curves under $f_c$ and $f_0$, see \cite[Lemma 6.1]{FF12}: if $n$ is the least integer with $c\in f_c^{-n}(U_c)$, then $\psi_c$ is extended to $f_c^{-n}(U_c)$, the extension is uniquely determined by $\psi_c|_{U_c}$, and it still satisfies $f_0 \circ \psi_c = \psi_c \circ f_c$ on the enlarged domain. We will use this functional equation for the extended map throughout; a description of the extended domain in terms of the Green's function of \cite{BF25} is given in the proof of Lemma \ref{lem:Psicont} below. The following definition therefore makes sense.

\begin{definition}
\label{def:Psi}
Let $K> 1$ and $\theta \in(-\pi/2,\pi/2]$. Then we define the map $\Psi :\C\setminus \mathcal{M}_{K,\theta} \to \C$ via
\[ \Psi (c) = \psi_c(c).\]
\end{definition}
The map $\Psi$ depends on $K$ and $\theta$ but we will suppress this dependency in the notation for simplicity.
We denote by $\mathcal{H}(A)$ the linear dilatation of a linear map $A$ and we denote by $D_u $ the derivative with respect to the variable $u$.
Our first main theorem is then as follows.

\begin{theorem}
\label{thm:annular}
There exists $R = R(K,\theta) < \infty$ such that $\Psi$ is a quasiconformal homeomorphism of $\{c : |c| >R \}$ onto its image with $\mathcal{H}(D_c \Psi) \leq 1 + O_{K,\theta}(|c|^{-1})$, $\det D_c \Psi > 0$, and $D_c \Psi \to I$ as $c\to \infty$.
\end{theorem}

This theorem shows that the Douady-Hubbard program works, at least in a neighbourhood of infinity. The map $\Psi$ satisfies the functional equation $f_0^n( \Psi(c)) = \psi_c (P_n(c))$ and this can, in principle, be used to extend the domain of definition of $\Psi$. However, there are obstructions that arise because $P_n(c)$ is no longer holomorphic, or even quasiregular, in $c$. We show in Proposition \ref{prop:derivative} below that $\Psi$ is $C^1$ on $\C \setminus \mathcal{M}_{K,\theta}$, so that $\det D_c \Psi$ is a continuous function there. 

We call $c_0 \in \C \setminus \mathcal{M}_{K,\theta}$ an \emph{orientation reversing parameter} if $\det D_c\Psi (c_0) <0$. If $U \subset \C\setminus \mathcal{M}_{K,\theta}$ is a connected open set on which $\det D_c \Psi$ takes both signs, we say that $\Psi$ has a \emph{fold} in $U$, and we call the non-empty set $\partial \{ \det D_c \Psi <0 \} \cap U$ the \emph{fold locus} in $U$. A fold has a strong consequence. A $C^1$ map which is injective on a connected open set is a homeomorphism onto its image by invariance of domain, so it has a constant local degree $\pm 1$ there, and this degree agrees with the sign of the Jacobian at every point where the Jacobian does not vanish. Hence $\Psi$ cannot be injective on any connected open set in which it has a fold.
In fact, we show that for a particular choice of parameters, $\Psi$ has orientation reversing parameters, which implies that folds must exist.

\begin{theorem}
\label{thm:fold}
Let $K=5$ and $\theta = 0$. Then $c=-1/10$ is an orientation reversing parameter, and $\Psi$ has a fold in every connected open subset $U$ of $\C \setminus \mathcal{M}_{5,0}$ which contains both $-1/10$ and a neighbourhood of infinity. In particular, $\Psi$ is not injective on $\C \setminus \mathcal{M}_{5,0}$, and so is not a quasiconformal map there.
\end{theorem}

Numerical investigation indicates that in fact $\Psi$ has orientation reversing parameters for every $K$ tested in the range $[1.015,5]$, with $\theta = 0$. We will discuss this further later on. Our conclusion is that any attempt to prove that $\mathcal{M}_{K,\theta}$ is connected cannot fully run through an argument involving the analogue of the uniformizing coordinate that Douady and Hubbard construct. Note, however, that this does not prove that $\mathcal{M}_{K,\theta}$ is not connected. This question remains open.

To return to positive results, we make the following definitions. The first is the natural ``orbit stays large'' condition; the second is the weaker, weighted condition that actually governs the orientation of $\Psi$.

\begin{definition}
\label{def:omegar}
For $r>1$, we set
\[ \Omega_r := \left \{ c \in \C \setminus \mathcal{M}_{K,\theta} : |P_i(c)| \geq r \text{ for all }i\geq 0\right \}.\]
\end{definition}

\begin{definition}
\label{def:sigma}
For $c\in \C$ and $n\geq 1$ we define the \emph{close-return sums}
\[ \sigma_n(c) := \sum_{j=1}^{n} \prod_{i=0}^{j-1} \frac{1}{2|h(P_i(c))|}, \qquad \sigma(c) := \lim_{n\to\infty} \sigma_n(c) \in (0,\infty],\]
with the convention that $\sigma_n(c) = \sigma(c) = \infty$ whenever $P_i(c) = 0$ for some $i\leq n-1$; such $c$ lie in $\mathcal{M}_{K,\theta}$, as $0$ is then periodic under $f_c$. We set
\[ \Omega^* := \{ c\in \C \setminus \mathcal{M}_{K,\theta} : \sigma(c) < 1 \}.\]
\end{definition}

The $j$-th term of $\sigma(c)$ is the natural bound for $||\Pi_j^{-1}||$, where $\Pi_j$ is the product of the derivatives of $f_c$ along the first $j$ points of the critical orbit, see \eqref{eq:DcPn} below. Here, we just note that the sum is finite for every $c\notin \mathcal{M}_{K,\theta}$ by Lemma \ref{lem:alg2} below. Since $|h(w)|\geq |w|$, if $c\in \Omega_r$ then $\sigma(c) \leq \sum_{j\geq 1} (2r)^{-j} = (2r-1)^{-1} <1$, so
\[ \Omega_r \subset \Omega^* \quad \text{for every } r>1.\]
The inclusion is far from an equality. If $c\in \Omega_r$ then $|c| = |P_0(c)| \geq r>1$, so $\Omega_r$ misses the whole unit disk, whereas $\Omega^*$ allows the sequence $P_n(c)$ to pass through the unit disk, even more than once, provided the products of the $(2|h(P_i(c))|)^{-1}$ stay summable below $1$. More precisely, orientation can only be reversed if the sequence $P_n(c)$ makes a quantitatively strong close return to $0$, in terms of $\sigma$, see Corollary \ref{cor:fold} below. We note also that neither $\Omega_r$ nor $\Omega^*$ is the same as $\{|c| >R \}$: the hypothesis of Theorem \ref{thm:annular} forces $|c|$ to be large from the start, whereas $\Omega_r$ and $\Omega^*$ only constrain the orbit, and it is possible for an orbit in $\Omega^*$ to hang around near the Mandelbrot set for a while before escaping.

In \cite{BF25}, Broderius and the author constructed an analogue of the Green's function for the maps $f_c$. This is a non-negative, continuous function $G_{c} :\C \to \R$ that is only non-zero on the escaping set $I(f_c)$ and satisfies $G_c(f_c(z)) = 2G_c(z)$ for all $z\in \C$. As a natural follow-up to this construction, we make the following definition.

\begin{definition}
\label{def:mandelgreen}
If $K>1$ and $\theta \in (-\pi/2,\pi/2]$ are fixed, then we define the function $G_{\mathcal{M}}$ via
\[ G_{\mathcal{M}} (c) := G_{c}(c).\]
\end{definition}

As the functions $G_c$ need not be harmonic, there is no reason to assume that $G_{\mathcal{M}}$ is harmonic. As $c\in I(f_c)$ for $c\notin \mathcal{M}_{K,\theta}$, it follows that $G_{\mathcal{M}}$ is strictly positive on $\C \setminus \mathcal{M}_{K,\theta}$.
We set
\begin{equation}
\label{eq:t0}
t_0 := \inf \left \{ t>0 : \{ G_{\mathcal{M}} >s \} \subset \Omega^* \text{ for all } s\geq t \right \},
\end{equation}
and, for comparison with the cruder region, 
\begin{equation}
\label{eq:tstar} 
t_*(r):= \inf \{ t>0 : \{ G_{\mathcal{M}} >s \} \subset \Omega_r \text{ for all } s\geq t \},
\end{equation}
for $r>1$. As $\Omega_r \subset \Omega^*$, we have $t_0 \leq \inf_{r>1} t_*(r)$, and we observe in Lemma \ref{lem:tstar} below that $t_*(r) \leq \log r + \log (K^2+1)$ for $r\geq 2$, so that $t_0 \leq \log 2 + \log(K^2+1)$ is finite. 

With these definitions in place, we may improve Theorem \ref{thm:annular} in terms of the Green's function.

\begin{theorem}
\label{thm:potential}
Let $K> 1$ and $\theta \in (-\pi/2,\pi/2]$, and let $t>t_0$. Then $\sigma_t := \sup \{ \sigma(c) : G_{\mathcal{M}}(c) >t \} <1$, and
\[ \Psi : \{ G_{\mathcal{M}} > t \} \to \{ G_{0} > t\}\]
is a quasiconformal homeomorphism with
\[ \det D_c\Psi > 0,\quad \mathcal{H}(D_c\Psi) \leq \frac{ 1+\sigma_t}{1-\sigma_t}\, e^{2s(t)},\quad s(t) \leq C(K,\theta) \left ( 1 + \log^+ \tfrac{1}{t} \right ) \max \left \{ 1, t^{-\log_2 K} \right \}.\]
In particular, if $t>t_*(r)$ for some $r>1$, then $\sigma_t \leq (2r-1)^{-1}$ and $\mathcal{H}(D_c\Psi) \leq \frac{r}{r-1}e^{2s(t)}$.
\end{theorem}

\begin{corollary}
\label{cor:potential}
For every $t>t_0$, the set $\{ G_{\mathcal{M}} \leq t \}$ is a closed topological disk containing $\mathcal{M}_{K,\theta}$, and $\mathcal{M}_{K,\theta}$ carries external parameter rays and equipotentials at all potentials greater than $t_0$.
\end{corollary}

We emphasize that the threshold is necessary. In Lemma \ref{lem:tzero} we show that $t_0 \geq G_{\mathcal{M}}(1/(2K)) >0$ for $\theta = 0$ and every $K\geq 1$, with $t_0 \geq 1/26$ for $1\leq K \leq 2$, so that $\{G_{\mathcal{M}} \leq t_0\}$ is a genuine collar around $\mathcal{M}_{K,\theta}$ on which Theorem \ref{thm:potential} does not apply; and Theorem \ref{thm:fold} shows that for $K=5$, $\theta=0$ the collar contains parameters at which the conclusion of Theorem \ref{thm:potential} actually fails. In particular, $t_0$ does not tend to $0$ as $K\to 1$ when $\theta =0$, so that Theorem \ref{thm:potential} does not recover the Douady-Hubbard theorem in the limit: the condition $\sigma <1$ is sufficient for orientation preservation, but it is far from necessary near $\mathcal{M}_{K,\theta}$. 

Finally, we point out that Bielefeld et al \cite{BSTV93} studied a similar family of maps $z\mapsto |z|^{2\alpha -2}z^2 +c$ for $c\in \C$ and $\alpha >\tfrac12$ real. These are compositions of $g_c(z) = z^2+c$ and the quasiconformal maps $h_{\alpha}(z) = z|z|^{\alpha-1}$. A quick computation shows that $\mu_h = \frac{\alpha-1}{\alpha+1} \cdot e^{2i\arg z}$ and so this family is different from the one discussed here. An analogous generalization of the Mandelbrot set is defined in \cite{BSTV93}, and the authors there note the difficulty of studying its properties. It is plausible the methods here could be generalized to that setting, and perhaps other families of planar quasiregular maps.

The paper is organized as follows. Section 2 collects preliminaries on quasiregular mappings, on the family $f_c$, and on the Green's function of \cite{BF25}, including the new formula $G_c(z) = \lim 2^{-n}\log^+|f_c^n(z)|$ and the structure of the equipotentials of the model map $f_0$. Section \ref{sec:bottreg} establishes the regularity of the B\"ottcher coordinate $\psi_c$ jointly in $(c,z)$ that we need, refining the construction in \cite{FF12}. Section \ref{sec:estimates} contains the estimates at the heart of the paper: orbit control and the continuity of the Green's function, the decomposition $D_cP_n = \Pi_nR_n$ and the close-return bound on $R_n$ (Lemma \ref{lem:Rnclose}), and the comparison between the parameter side and the dynamical side which produces the factor $S_\infty$. These are assembled in Proposition \ref{prop:derivative}, which shows that $\Psi$ is $C^1$ on $\C \setminus \mathcal{M}_{K,\theta}$ and gives the formula $D_c\Psi = \lambda_\infty S_\infty R_\infty$ on which everything else rests. Theorem \ref{thm:annular} is proved in Section 5. Section \ref{sec:fold} contains the Frozen Sign Lemma and the proof of Theorem \ref{thm:fold}. Theorem \ref{thm:potential} and Corollary \ref{cor:potential} are proved in Section \ref{sec:green}.

{\bf Acknowledgements:} a preliminary version of a result close to Theorem \ref{thm:annular} was first developed in 2013. It subsequently stayed in development with the hope that the uniformizing coordinate could be pushed all the way to $\partial \mathcal{M}_{K,\theta}$. In 2026, numerical investigations with Claude Opus 5 indicated for the first time that orientation reversing parameters did exist. This, and the recent work with Broderius \cite{BF25}, led to the development of Theorem \ref{thm:potential} addressing the question of just how the uniformizing coordinate could be pushed. The fractal images were produced with {\it Ultra Fractal 6}.

\section{Preliminaries}

\subsection{Quasiregular mappings}

A quasiconformal mapping $f:\C \to\C$ is a homeomorphism such that $f$ is in the Sobolev space $W^1_{2,loc}(\C)$ and there exists $k\in [0,1)$ such that the complex dilatation $\mu_f = f_{\zbar}/f_z$ satisfies
\[ |\mu_f(z)| \leq k \]
almost everywhere in $\C$. The dilatation of $f$ at $z\in \C$ is
\[ K_f(z) := \frac{1+|\mu_f(z)|}{1-|\mu_f(z)|} \geq 1.\]
A mapping is called $K$-quasiconformal if $K_f(z) \leq K$ almost everywhere. The smallest such constant is called the maximal dilatation and denoted by $K(f)$. If we drop the assumption on injectivity, then $f$ is called a quasiregular mapping.
We refer to the books of Iwaniec and Martin \cite{IM01} and Rickman \cite{Rickman} for much more on the development of the theory of quasiregular mappings.

If $A$ is a linear map, then its operator norm is $||A|| = \max_{|z|=1} |A(z)|$, and $\ell(A)$ is defined by $\ell(A) = \min_{|z|=1} |A(z)|$. The linear dilatation of $A$ is then $\mathcal{H}(A) = ||A|| / \ell(A)$. We record for later use the elementary facts that $\mathcal{H}$ is submultiplicative, that $\mathcal{H}(A) = \mathcal{H}(A^{-1})$, that $\mathcal{H}(\lambda A) = \mathcal{H}(A)$ for $\lambda >0$, and that $\ell$ is supermultiplicative. A quasiregular map is differentiable almost everywhere, and the quasiregularity definition implies that $\mathcal{H}(f')$ is uniformly bounded everywhere $f'$ exists. For planar quasiregular mappings, the linear dilatation agrees with the maximal dilatation.

In the plane, we have the following two crucial results. First, there is a surprising correspondence between quasiconformal mappings and measurable functions, see for example \cite[p.8]{IM01}.

\begin{theorem}[Measurable Riemann Mapping Theorem]
Suppose that $\mu \in L^{\infty}(\C)$ with $||\mu||_{\infty} \leq k <1$. Then there exists a quasiconformal map $f:\C \to \C$ with complex dilatation equal to $\mu$ almost everywhere. Moreover, $f$ is unique if it fixes $0,1$ and $\infty$.
\end{theorem}

Every planar quasiregular mapping has an important factorization, see for example \cite[p.254]{IM01}.

\begin{theorem}[Stoilow Factorization Theorem]
\label{thm:stoilow}
Let $f:\C\to \C$ be a quasiregular mapping. Then there exists a holomorphic function $g$ and a quasiconformal mapping $h$ such that $f=g\circ h$.
\end{theorem}

In fact, in the plane, some sources use this decomposition as the definition of a quasiregular mapping. However, this approach does not generalize to higher dimensions, so even though this paper is two dimensional, we keep the Stoilow Factorization as a theorem.

\subsection{The mappings $f_c$}

As always throughout the paper, we fix $K>1$ and $\theta \in (-\pi/2 , \pi/2]$. Then we define $h:=h_{K,\theta}$ to be the stretch by factor $K$ in the direction of $e^{i\theta}$. It is an elementary computation (see for example \cite{BF25}) to check that
\begin{align*}
h_{K,\theta}(z) &= \left ( \frac{K+1}{2} \right ) z + e^{2i\theta} \left (\frac{K-1}{2} \right ) \zbar \\
&= \left [ x(K\cos^2(\theta)+\sin^2(\theta))+y(K-1)\sin(\theta)\cos(\theta) \right ] \\
& \hskip0.5in + i\left [ x(K-1)\cos(\theta)\sin(\theta)+y(K\sin^2(\theta)+\cos^2(\theta))\right ].
\end{align*}
It is clear that
\[ \mu_{h_{K,\theta}} \equiv e^{2i\theta} \left ( \frac{K-1}{K+1} \right ),\]
and it follows from the Measurable Riemann Mapping Theorem and \cite[Proposition 3.1]{FG10} that any quasiconformal mapping of $\C$ with constant complex dilatation given as above arises by post-composing $h_{K,\theta}$ with a complex affine map, since the composition of any such map with $h_{K,\theta}^{-1}$ is a conformal automorphism of $\C$. The singular values of $h_{K,\theta}$, viewed as a real-linear map, are $K$ and $1$, so that
\begin{equation}
\label{eq:hsv}
||h|| = K, \quad \ell(h) = 1, \quad \mathcal{H}(h) = K, \quad \det h = K, \quad |h(z)| \geq |z| \text{ for all } z \in \C.
\end{equation}

If we view quasiconformal mappings with constant complex dilatation as the simplest in their class, then we can view quadratic polynomials as the simplest non-injective holomorphic functions. In light of the Stoilow Factorization Theorem, the simplest quasiregular mappings arise as compositions of quadratic polynomials and mappings of the form $h_{K,\theta}$. As defined in the introduction, for $K>1$, $\theta \in (-\pi/2,\pi/2]$ and $c\in \C$, we have $f_c= g_c \circ h_{K,\theta}$.

\subsection{The dynamics of $f_c$ and properties of $G_c$}

Iterates of quasiregular mappings are again quasiregular, but the maximal dilatation typically goes up. If there is a uniform bound on the maximal dilatation of the iterates of $f:\C\to\C$, then it is well-known that $f$ is quasiconformally conjugate to a holomorphic function. Thus, to be of independent interest, it is important to know from \cite[Theorem 2.6]{FF12} that the mappings $f_c$ are not uniformly quasiregular.

The escaping set for $f_c$ is a non-empty, open neighborhood of $\infty$ by \cite[Theorem 4.3]{FG10}. It follows that the bounded orbit set $BO(f_c)$ is the complement of the escaping set. All of $I(f_c), BO(f_c)$ and $\partial I(f_c)$ are completely invariant under $f_c$.
We can ask for a conjugation of $f_c$ to a simpler mapping in a neighborhood of infinity, analogous to B\"ottcher's Theorem. The following result yields this.

\begin{theorem}[\cite{FF12}, Theorem 2.1]
\label{thm:bottcher}
Let $K>1$, $\theta \in (-\pi/2,\pi/2]$ and $c\in \C$. Then there exist a neighborhood of infinity $U = U_{K,\theta,c}$ and a quasiconformal map $\psi_c = \psi_{K,\theta,c}$ defined in $U$ such that
\[ \psi_c \circ f_c = f_0 \circ \psi_c \]
holds in $U$.
\end{theorem}

We can even semi-conjugate $f_c$ to multiplication by $2$ via the Green's function constructed in \cite{BF25}.

\begin{theorem}[\cite{BF25}, Theorem 1.1]
\label{thm:greens}
Let $K>1$ and $\theta \in (-\pi/2,\pi/2]$. For $c\in \C$, there exists a non-negative, continuous function $G_c : \C \to \R$ that is identically zero on $BO(f_c)$, non-zero on $I(f_c)$ and such that $G_c(f_c(z)) = 2G_c(z)$ for all $z\in \C$.
\end{theorem}

In \cite{BF25}, $G_c$ is constructed geometrically, first for $G_0$ directly and then using the B\"ottcher coordinate to generalize to $G_c$. We point out here that in fact $G_c$ can be obtained via a formula familiar to complex dynamics. The proof uses Proposition \ref{prop:bottreg}(d) from Section \ref{sec:bottreg}, but there is no circularity, as nothing in Section \ref{sec:bottreg} depends on the results of this subsection.

\begin{proposition}
\label{prop:bottformula}
Let $K>1$, $\theta \in (-\pi/2, \pi/2 ]$ and $c\in \C$. Then
\[ G_c(z) = \lim_{n\to \infty} 2^{-n} \log^+ |f_c^n(z)|.\]
\end{proposition}

\begin{proof}
Consider first $c=0$. By \cite[Definitions 3.3 and 3.8]{BF25}, we have $G_0 = \log |\tau_{K,\theta,0}|$, where $\tau_{K,\theta,0}(z) = z / b_{K,\theta}(\arg z)$ and $b_{K,\theta}(\phi)$ is the unique element of $\partial I(f_0)$ with argument $\phi$ (see \cite[Definition 3.2 and Lemma 3.1]{BF25}). As $G_0(f_0^n(z)) = 2^nG_0(z)$, we obtain
\[ G_0(z) = 2^{-n} \left ( \log |f_0^n(z)| - \log |b_{K,\theta}(\arg f_0^n(z)) | \right )\]
for every $n$. By \cite[Lemma 2.3]{BF25}, the ball of radius $(2K^2)^{-1}$ about the origin lies in $BO(f_0)$, and $I(f_0)$ contains a neighbourhood of infinity, so $|b_{K,\theta}|$ is bounded above and below by positive constants. Letting $n\to \infty$, the second term above tends to $0$, and hence
\[ G_0(z) = \lim_{n\to \infty} 2^{-n} \log |f_0^n(z)|.\]
For general $c$, we have $G_c = G_0 \circ \psi_c$ by \cite[Definitions 3.6 and 3.8]{BF25}, and so in a neighbourhood of infinity,
\[ G_c(z) = \lim_{n\to \infty} 2^{-n} \log |\psi_c ( f_c^n(z))|.\]
As $\psi_c(w) = w + O(|c|/|w|)$ by Proposition \ref{prop:bottreg}(d) below, we have $\log |\psi_c(w)| = \log|w| + o(1)$ as $|w| \to \infty$, and so $\psi_c$ disappears in the limit. Finally, both $G_c$ and $z\mapsto \lim_n 2^{-n}\log^+|f_c^n(z)|$ vanish on $BO(f_c)$ and satisfy the functional equation $u\circ f_c = 2u$, so the identity propagates from a neighbourhood of infinity to all of $\C$ by pulling back, exactly as in the proof of \cite[Theorem 1.1]{BF25}.
\end{proof}

We will use the following growth estimate of $G_c$ later on.

\begin{corollary}
\label{cor:bott}
For all $z,c\in \C$ and $K\geq 1$, we have
\[ G_c(z) \leq \log \max ( |z|,|c|,1) + \log (K^2+1).\]
\end{corollary}

\begin{proof}
We have the elementary estimate $|f_c(z)| \leq K^2|z|^2+|c|$. Set $\Theta_0 = \max (|z|,|c|,1)$ and $\Theta_{n+1} = (K^2+1)\Theta_n^2$, and write $z_n = f_c^n(z)$. If $|z_n| \leq \Theta_n$ and $\Theta_n \geq \max (|c|,1)$, then
\[ |z_{n+1}| \leq K^2\Theta_n^2 + |c| \leq K^2\Theta_n^2 + \Theta_n^2 = \Theta_{n+1},\]
using $\Theta_n \leq \Theta_n^2$, and moreover $\Theta_{n+1} \geq \Theta_n \geq \max(|c|,1)$. Hence both hypotheses propagate forward inductively. Taking logarithms, $a_n:= 2^{-n} \log \Theta_n$ satisfies
\[ a_{n+1} = a_n + 2^{-(n+1)} \log (K^2+1),\]
so $a_n \leq \log \Theta_0 + \log (K^2+1)$. Applying Proposition \ref{prop:bottformula}, we see that
\[ G_c(z) = \lim_{n\to \infty} 2^{-n} \log^+ |f_c^n(z)| \leq \lim_{n\to \infty} a_n \leq \log \Theta_0 + \log (K^2+1),\]
as needed.
\end{proof}

We will also need a lower bound for $G_c$, valid once the orbit is large. This is postponed to Lemma \ref{lem:Glower}, since it uses the orbit control of Lemma \ref{lem:Pescape}.

We finish this section by recording the structure of the equipotentials of the model map $f_0$. This is a place where $f_0$ behaves markedly better than $f_c$ for $c\notin \mathcal{M}_{K,\theta}$, as by \cite[Theorem 1.2]{BF25}, the set $\partial I(f_c)$ has uncountably many components when $c\notin \mathcal{M}_{K,\theta}$, whereas $0 \in BO(f_0)$ and so $0 \in \mathcal{M}_{K,\theta}$.

\begin{lemma}
\label{lem:jordan}
The map $b_{K,\theta} : [0,2\pi) \to \C$ is continuous and injective, $\partial I(f_0) = b_{K,\theta}([0,2\pi))$ is a Jordan curve, and for every $t>0$ we have
\begin{equation}
\label{eq:equipot}
\{ z: G_0(z) =t \} = e^t \partial I(f_0).
\end{equation}
In particular, $\{ z : G_0(z) > t \}$ is connected for every $t>0$.
\end{lemma}

\begin{proof}
By \cite[Lemma 3.1]{BF25}, $\partial I(f_0) \cap R_{\phi}$ consists of the single point $b_{K,\theta}(\phi)$, where $R_{\phi}$ denotes the ray $\{te^{i\phi} : t\geq 0\}$; injectivity of $b_{K,\theta}$ is immediate since $\arg b_{K,\theta}(\phi) = \phi$. As noted in the proof of Proposition \ref{prop:bottformula}, $|b_{K,\theta}|$ is bounded above and below by positive constants. For continuity, suppose $\phi_n \to \phi$. The sequence $(b_{K,\theta}(\phi_n))$ is bounded, and if $L$ is any limit point then $L \in \partial I(f_0)$, since $\partial I(f_0)$ is closed, and $L \in R_{\phi}$ and so, by uniqueness, $L = b_{K,\theta}(\phi)$. Hence $b_{K,\theta}(\phi_n) \to b_{K,\theta}(\phi)$. A continuous injection of the circle into $\C$ is a homeomorphism onto its image, so $\partial I(f_0)$ is a Jordan curve.

By \cite[Lemma 3.4]{BF25} we have $G_0(rz) = \log r + G_0(z)$ for $r>0$, and $G_0$ vanishes precisely on $BO(f_0)$, so $G_0 = 0$ on $\partial I(f_0)$. Therefore $G_0(re^{i\phi}) = \log ( r / |b_{K,\theta}(\phi)| )$, which gives \eqref{eq:equipot} and
\[ \{ z: G_0(z) > t \} = \{ re^{i\phi} : r > e^t |b_{K,\theta}(\phi)| \}.\]
This set is connected: any two of its points may be joined by moving radially outwards to a common circle $\{|z| = \varrho\}$ with $\varrho > e^t \max_{\phi} |b_{K,\theta}(\phi)|$, and then along that circle.
\end{proof}

\section{Properties of the B\"ottcher coordinate}
\label{sec:bottreg}

To perform our analysis, we will need further properties of the B\"ottcher coordinate from Theorem \ref{thm:bottcher}. More precisely, we will need the following result.

\begin{proposition}
\label{prop:bottreg}
There is a constant $a = a(K,\theta) \geq 1$ such that, on the set
\[ \mathcal{D} := \left \{ (c,z) \in \C \times \C : |z| \geq a\sqrt{1+|c|} \right \}, \]
the B\"ottcher coordinate $\psi_c(z)$ of Theorem \ref{thm:bottcher} is defined, with no extension needed, and has the following properties, with all implied constants depending only on $K$ and $\theta$.
\begin{enumerate}[(a)]
\item The map $(c,z) \mapsto \psi_c(z)$ is jointly $C^1$.
\item The $c$-partial derivatives of $\psi_c$ satisfy
\[ \partial_c \psi_c(z) = O(|z|^{-1}),\quad \partial_{\overline{c}} \psi_c(z) = O(|z|^{-1}),\]
as $|z|\to \infty$, with implied constant independent of $c$.
\item The $z$-partial derivatives of $\psi_c$ satisfy
\[ \partial_z \psi_c(z) = 1+ O(|c|/|z|^2), \quad \partial_{\overline{z}} \psi_c(z) = O(|c|/|z|^2),\]
as $|z| \to\infty$.
\item There is a constant $C' = C'(K,\theta)$ such that
\[ |\psi_c(z) - z| \leq C' \frac{|c|}{|z|}.\]
\end{enumerate}
\end{proposition}

\subsection{The set-up}

To establish these properties, we will need to perform a finer analysis of the work in \cite{FF12}. We will not completely re-do the content of that paper, but some overlap is necessary. For consistency, we will aim to keep as much of the notation of \cite{FF12} as possible.

To start, if $f$ is defined in a neighbourhood of infinity $\{ z : |z| >R \}$ and has no zeros, then its logarithmic transform is defined by $\ft (X) = \log f (e^X)$ in the right half-plane $\{ X: \Re X > \log R \}$ for an appropriate branch of the logarithm. We will work in the half-plane $L := \{ X:\Re X > \sigma \}$, for an appropriately chosen $\sigma >0$.

Recalling that $g_c(z) = z^2+c$, it follows by a computation \cite[Lemma 3.11]{FF12} that $\widetilde{g_c}(X) = 2X + \log (1+ce^{-2X})$.
From \cite[Lemma 5.5]{FF12}, the logarithmic transform of $f_c$ satisfies
\begin{equation}
\label{eq:l55}
| \Re \widetilde{f}(X) - 2\Re X| \leq C_6,
\end{equation}
for some constant $C_6$. This constant $C_6$ is the notation from \cite{FF12}; it depends only on $K$ and $\theta$ once $|ce^{-2X}| \leq \tfrac12$ on $L$, which we arrange below. We will always assume $\sigma > C_6$, so that $\widetilde{f}(L) \subset L$. For simplicity, we make a notational change from \cite{FF12}: we write $C$ for any constant depending only on $K$, $\theta$ and, where a parameter $\alpha \in (1,2)$ is present, on $\alpha$, but not on $c$, $X$ or $\sigma$. The specific value of $C$ may change at each step. Several of the inductions below are completed only if $\sigma$ is large enough, and each such requirement involves only $K$, $\theta$ and $\alpha$; we fix once and for all a constant $\sigma_0 = \sigma_0(K,\theta,\alpha) > C_6$ satisfying all of them, and take $\sigma \geq \sigma_0$.

For convenience, we set
\begin{equation}
\label{eq:rhoc}
\rho(X,c) = \log(1+ce^{-2X}).
\end{equation}
By \cite[Lemma 3.12]{FF12}, the logarithmic transforms of $h = h_{K,\theta}$ and its inverse are
\[ \widetilde{h}(X) = X + \log \left [ \left( \frac{K+1}{2}\right ) + e^{2i\theta - 2i\Im X} \left (\frac{K-1}{2}\right ) \right ] \]
and
\[ \widetilde{h^{-1}}(X) = X + \log \left [ \left ( \frac{K+1}{2K} \right ) - e^{2i\theta - 2i\Im X} \left ( \frac{K-1}{2K} \right ) \right ]. \]
Again for convenience, we set $\varphi(X) = \widetilde{h}(X) - X$ and $\xi (X) = \widetilde{h^{-1}}(X) - X$.

The B\"ottcher coordinate is constructed by defining a sequence $(\psi_k)_{k=1}^{\infty}$, where $\psi_1 = f_0^{-1} \circ f_c$ for a branch of the inverse of $f_0$ so that $\psi_1(z) = z+o(1)$ as $|z|\to \infty$. Inductively, $\psi_{k+1}(z) = f_0^{-1} \circ \psi_k \circ f_c$ with the same choice for the branch of the inverse. Then $\psi_k$ converges to $\psi_c$. It turns out to be more amenable to perform computations in the logarithmic transform, so we set $F_k  = \widetilde{\psi_k}$ in $L$. Then
\begin{equation}
\label{eq:Fk+1}
F_{k+1} = \widetilde{h^{-1}} \circ \widetilde{S} \circ F_k \circ \widetilde{g_c} \circ \widetilde{h},
\end{equation}
where $\widetilde{S} (X) = X/2$ is the logarithmic transform of the square root. We also write $F_k(X) = X +T_k(X)$, so that $T_0 \equiv 0$.

We note that as we will be estimating based on expanding the logarithm term in \eqref{eq:rhoc}, we will need $|ce^{-2X}| \leq \tfrac12$, so we will need to take $\sigma > \tfrac12 \log (2|c|)$. Together with the requirement $\sigma \geq \sigma_0$, this means that we work on the half-plane $L = \{ X : \Re X > \sigma \}$ with
\begin{equation}
\label{eq:sigma}
\sigma = \sigma(c) := \max \left \{ \sigma_0, \tfrac12 \log (2|c|) \right \},
\end{equation}
which in the $z$-variable is the domain $\{ |z| > e^{\sigma} \}$. Setting $a(K,\theta) := 2\max \{ e^{\sigma_0}, 1\}$, this domain contains
\begin{equation}
\label{eq:A}
|z| \geq a(K,\theta) \sqrt{1+|c|},
\end{equation}
which is the set $\mathcal{D}$ of Proposition \ref{prop:bottreg}. We stress that although $\sigma$ depends on $c$, the constants $C$ do not.

\subsection{Introducing $c$ into the picture and proving~(d)}

At this point, we wish to make the $c$-dependence of the functions here clearer. As $\widetilde{h}$, $\widetilde{h^{-1}}$ and $\widetilde{S}$ only depend on $K$ and $\theta$, the only place $c$ enters the picture is with $\widetilde{g_c}$. By \eqref{eq:rhoc}, we have
\[ | \rho ( X+\varphi(X) , c)| \leq C|c| e^{-2\Re X} .\]
Then the proof of \cite[Lemma 5.6]{FF12} goes through verbatim with this extra $|c|$ factor to give
\begin{equation}
\label{eq:T1}
|T_1(X)| \leq C|c| e^{-2\Re X}.
\end{equation}
Here, we observe that the conclusion of \cite[Lemma 5.7]{FF12} may actually be strengthened. There the estimate $|T_k(X)| < C_8 e^{-\alpha \Re X}$ was given for some $\alpha \in(1,2)$. However, a closer analysis of that proof shows that in fact it holds for $\alpha =2$. In the language of that proof, all that is needed is that $C_8$ is chosen larger than $(1+C_2)C_5$ and that $\sigma_0$ is chosen so that $(1+C_2)e^{2C_6}e^{-2\sigma_0} <1$; the same applies to the version with the extra $|c|$ factor, since $|c|$ multiplies both $C_5$ and $C_8$. Combining this minor upgrade with \eqref{eq:T1} shows that
\begin{equation}
\label{eq:Tkbound}
|T_k(X)| \leq C |c| e^{-2 \Re X}
\end{equation}
for all $k\in \N$. The proof of \cite[Lemma 5.8]{FF12} then goes through as written there, for any fixed $\alpha \in (1,2)$, with the conclusion that
\[ |F_{k+1}(X) - F_k(X) | \leq C|c| e^{-\alpha^k \Re X} \]
for all $k\geq 0$. This is used to establish that $T_k$ converges to $T$ uniformly for $X \in L$ and, moreover, $|T(X)| \leq C |c| e^{-2 \Re X}$.
Undoing the logarithmic transform exactly as in \cite[Section 5.4]{FF12}, and setting $z=e^X$ yields
\begin{align*}
|\psi_c(z)-z| &=| \exp ( \log z + T(\log z) ) - z | \\
&= | z | \cdot | \exp [T(\log z)] -1| \\
& \leq  |z| |T(\log z) | e^{|T(\log z)|}\\
&\leq C|c| |z|^{-1} ,
\end{align*}
using $|e^{w}-1| \leq |w|e^{|w|}$ and the fact that $|T| \leq C|c|e^{-2\sigma} \leq C$ is bounded on $L$ by \eqref{eq:sigma}. This proves Proposition \ref{prop:bottreg}~(d).

\subsection{Partial derivatives with respect to $z$ and proving~(c)}

By \eqref{eq:rhoc}, we have
\[ \rho_X = \frac{-2ce^{-2X}}{1+ce^{-2X}}, \quad \rho_c = \frac{e^{-2X}}{1+ce^{-2X}},\quad \rho_{\overline{X}} = \rho_{\overline{c}} \equiv 0.\]
We set $\rho_1(X,c) = \rho(X+\varphi(X), c)$. Then modifying \cite[Lemma 5.10]{FF12} as we did in the previous subsection, it follows that
\begin{equation}
\label{eq:rho1X}
|\partial_X (\rho_1)| \leq C |c| e^{-2\Re X}, \quad |\partial_{\overline{X}} (\rho_1)| \leq C|c|e^{-2 \Re X}.
\end{equation}
We will also note for later that $\rho$ is holomorphic in $c$, and hence so is $\rho_1$, since the argument $X+\varphi(X)$ does not depend on $c$. Consequently
\begin{equation}
\label{eq:rho1c}
|\partial_c (\rho_1)| \leq Ce^{-2\Re X},\quad \partial_{\overline{c}}(\rho_1) \equiv 0.
\end{equation}
We emphasize that the bound in \eqref{eq:rho1c} carries no factor of $|c|$, in contrast with \eqref{eq:rho1X}. This is the reason the implied constant in Proposition \ref{prop:bottreg}~(b) may be taken independent of $c$.

Now, from \eqref{eq:Fk+1}, we can write
\[ F_{k+1}(X) = \frac{F_k(\widetilde{f}(X))}{2} + \xi \left ( \frac{ F_k(\widetilde{f}(X))}{2} \right ).\]
Just as in \cite[(5.10)]{FF12}, we write
\begin{equation}
\label{eq:Pk}
W_k(X) = \frac{F_k(\widetilde{f}(X))}{2} = X + \varphi(X) + \frac{\rho_1(X)}{2} + \frac{T_k(\widetilde{f}(X))}{2},
\end{equation}
so that $F_{k+1} = W_k + \xi(W_k)$. We set $F_{k+1} - F_k = \Delta_k + [ \xi(W_k) - \xi(W_{k-1}) ]$, where
\begin{equation}
\label{eq:deltak}
\Delta_k = W_k - W_{k-1} = \tfrac12 \left ( T_k(\widetilde{f}) - T_{k-1}(\widetilde{f}) \right ).
\end{equation}
From here, the proof of \cite[Proposition 5.11]{FF12} essentially goes through word for word, with the extra $|c|$ factor and with $\alpha = 2$, to yield
\begin{equation}
\label{eq:pXTk}
| \partial_X (T_k)| \leq C |c|e^{-2\Re X},\quad |\partial_{\overline{X}}(T_k)| \leq C |c| e^{-2\Re X}
\end{equation}
for all $k\in \N$. Then
\[ |\partial_X (F_k) - 1| \leq C|c|e^{-2\Re X}, \quad |\partial_{\overline{X}} (F_k)| \leq C|c|e^{-2\Re X}.\]
Undoing the logarithmic transform yields part~(c).

We record separately the estimate we shall need for $T_k$ evaluated at $\widetilde{f}(X)$. This is not the conclusion of \cite[Lemma 5.12]{FF12}, which bounds the derivative of the composite $X \mapsto T_k(\widetilde{f}(X))$; rather, it is the intermediate step in the proof of that lemma, namely \eqref{eq:pXTk} applied at the point $\widetilde{f}(X) \in L$, combined with \eqref{eq:l55}. Writing $\partial_w$ for the derivative of $T_k$ with respect to its own argument, we obtain
\begin{equation}
\label{eq:Tkpartials}
|\partial_w T_k ( \widetilde{f}(X)) | \leq C |c| e^{-4\Re X}, \quad | \partial_{\overline{w}} T_k(\widetilde{f}(X)) | \leq C |c| e^{-4 \Re X}
\end{equation}
for all $k\in \N$.

We will also need the following $X$-derivative version of \cite[Lemma 5.8]{FF12}, which does not appear in that paper.

\begin{lemma}
\label{lem:Fkdiff}
Given $\alpha\in (1,2)$, for all $k\geq 0$ and all $X \in L$ we have
\[ |\partial_X (F_{k+1} - F_k) | \leq C |c| e^{-\alpha^k \Re X},\quad |\partial_{\overline{X}} (F_{k+1} - F_k) | \leq C |c| e^{-\alpha^k \Re X} .\]
\end{lemma}

\begin{proof}
For the case $k=0$, we have $F_1(X) - F_0(X) = T_1(X)$. Then as $e^{-2\Re X} \leq e^{-\alpha \Re X}$, the base case follows from \eqref{eq:pXTk}.
Now assume the claimed inequalities hold for $k-1 \geq 0$. Recall that $F_{k+1} - F_k = \Delta_k + [ \xi(W_k) - \xi(W_{k-1}) ]$, so that, using \cite[Lemma 5.9]{FF12},
\begin{align*}
\partial_X (F_{k+1} - F_k) &= \left ( 1 + \partial_X \xi (W_k) \right ) \partial_X \Delta_k + \partial_{\overline{X}} \xi (W_k) \overline{ \partial_{\overline{X}} \Delta_k} \\
& \hskip0.3in + \left [ \partial_X \xi (W_k) - \partial_X \xi(W_{k-1}) \right ] \partial_X W_{k-1} \\
& \hskip0.3in + \left [ \partial_{\overline{X}} \xi (W_k) - \partial_{\overline{X}} \xi(W_{k-1}) \right ] \overline{\partial_{\overline{X}} W_{k-1}}.
\end{align*}
By the inductive hypothesis, applied at the point $\widetilde{f}(X) \in L$, together with \eqref{eq:l55},
\[ |\partial_w (F_k - F_{k-1})(\widetilde{f}(X)) | \leq C|c| e^{-\alpha^{k-1}\Re \widetilde{f}(X) } \leq C|c|e^{C_6\alpha^{k-1}}e^{-2\alpha^{k-1} \Re X}.\]
As $|\partial_X \widetilde{f}|$ and $|\partial_{\overline{X}} \widetilde{f} |$ are bounded by \cite[Lemma 5.4]{FF12} and \eqref{eq:rho1X}, it follows from \eqref{eq:deltak} that
\[ |\partial_X \Delta_k| \leq C|c| e^{\alpha^{k-1}(C_6 - (2-\alpha)\sigma)} e^{-\alpha^k \Re X},\]
and similarly for $\partial_{\overline{X}}\Delta_k$. The partial derivatives of $\xi$ are bounded above by \cite[Lemma 5.4]{FF12}, and the differences of the partial derivatives of $\xi$ are bounded by $C|\Delta_k|$, again by \cite[Lemma 5.4]{FF12}, where $|\Delta_k| \leq C|c|e^{\alpha^{k-1}C_6}e^{-2\alpha^{k-1}\Re X}$ by \eqref{eq:deltak}, \eqref{eq:l55} and the version of \cite[Lemma 5.8]{FF12} recorded in the previous subsection; the factors $\partial_X W_{k-1}$ and $\partial_{\overline{X}} W_{k-1}$ are bounded by \eqref{eq:Pk}, \eqref{eq:pXTk} and \eqref{eq:rho1X}. All terms are therefore bounded by $C|c| e^{\alpha^{k-1}(C_6 - (2-\alpha)\sigma_0)}e^{-\alpha^k \Re X}$, and since $\sigma_0$ was chosen so that $Ce^{\alpha^{k-1} ( C_6 - (2-\alpha) \sigma_0)} \leq 1$ for all $k\geq 1$ (this is one of the requirements on $\sigma_0$), the inductive step is complete. A similar argument runs for the $\overline{X}$-derivatives. 
\end{proof}

\subsection{Partial derivatives with respect to $c$ and proving~(b) and~(a)}

This is the point where we fully diverge from the work in \cite{FF12}. First, if $\eta$ is a function of a complex variable, with no explicit dependence on $c$, and it is composed with $w=W(c)$, then the Chainr Rule for complex derivatives gives
\begin{equation}
\label{eq:eta}
\partial_c ( \eta (W)) = [\partial_w \eta (W)] \partial_c W + [\partial_{\overline{w}} \eta(W)] \overline{\partial_{\overline{c}} W}, \quad
\partial_{\overline{c}} ( \eta (W)) = [\partial_w \eta (W)] \partial_{\overline{c}}W + [\partial_{\overline{w}} \eta (W)] \overline{\partial_c W}.
\end{equation}

Using \eqref{eq:eta} with $w= W_k$, \eqref{eq:Pk}, $\partial_c \widetilde{f} = \partial_c \rho_1$ and $\partial_{\overline{c}} \widetilde{f} = 0$, the partial derivatives of $W_k$ with respect to $c$ are
\begin{equation}
\label{eq:Ppartials}
\partial_c W_k = \tfrac12 \partial_c \rho_1 + \tfrac12 \left [ \partial_c T_k (\widetilde{f}) + \partial_w T_k (\widetilde{f}) \partial_c \rho_1 \right],\quad
\partial_{\overline{c}} W_k =  \tfrac12 \left [  \partial_{\overline{c}} T_k (\widetilde{f}) + \partial_{\overline{w}} T_k (\widetilde{f}) \overline{\partial_c \rho_1} \right ].
\end{equation}
We note that $\partial_{\overline{c}}W_k$ contains no forcing term involving $\rho_1$. Since $\rho_1$ is holomorphic in $c$, all antiholomorphic dependence on the parameter is generated by the non-conformality of $h$, through $\varphi$ and $\xi$ alone, and vanishes identically when $K=1$.

We may now run an induction on $k$, using \eqref{eq:rho1c}, \eqref{eq:Tkpartials} and \eqref{eq:Ppartials}, to obtain
\begin{equation}
\label{eq:cTk}
|\partial_c T_k| \leq C e^{-2\Re X}, \quad |\partial_{\overline{c}} T_k | \leq Ce^{-2\Re X}
\end{equation}
for all $k\in \N$. Indeed, the case $k=0$ is trivial as $T_0 \equiv 0$. Assuming \eqref{eq:cTk} for $k$, \eqref{eq:Ppartials} together with \eqref{eq:rho1c}, \eqref{eq:l55} and \eqref{eq:Tkpartials} gives $|\partial_c W_k| \leq Ce^{-2\Re X}$ and $|\partial_{\overline{c}}W_k| \leq Ce^{-4\Re X}$. Since $F_{k+1} = W_k + \xi(W_k)$, \eqref{eq:eta} gives
\[ \partial_c F_{k+1} = \left ( 1 + \partial_w \xi (W_k) \right ) \partial_c W_k + \partial_{\overline{w}} \xi(W_k) \overline{\partial_{\overline{c}} W_k},\]
and the analogous formula for $\partial_{\overline{c}}F_{k+1}$. As the partial derivatives of $\xi$ are bounded above by \cite[Lemma 5.4]{FF12}, and as $T_{k+1}(X) = F_{k+1}(X) - X$ has the same $c$-derivatives as $F_{k+1}$, this completes the induction, with the following observations. Write $C_\xi$ for a bound on the first partial derivatives of $\xi$, $C_\rho$ for the constant in \eqref{eq:rho1c}, and let $C_k$ be a common bound for both $c$-derivatives, so that $|\partial_c T_k| \leq C_k e^{-2\Re X}$ and $|\partial_{\overline{c}} T_k| \leq C_k e^{-2\Re X}$ on $L$; the two derivatives must be tracked together, since $\partial_c F_{k+1}$ involves $\partial_{\overline{c}}W_k$ and vice versa. The terms $\tfrac12 \partial_c T_k(\widetilde{f})$ and $\tfrac12 \partial_{\overline{c}} T_k(\widetilde{f})$ in \eqref{eq:Ppartials} each contribute at most $\tfrac12 C_k e^{2C_6}e^{-2\sigma_0} \cdot e^{-2\Re X}$ by \eqref{eq:l55}, while the terms involving $\partial_w T_k(\widetilde{f})$ and $\partial_{\overline{w}}T_k(\widetilde{f})$ contribute at most $C|c|e^{-6\Re X} \leq Ce^{-4\Re X} \leq Ce^{-2\sigma_0}e^{-2\Re X}$ by \eqref{eq:Tkpartials}, using $e^{-2\Re X} \leq (2|c|)^{-1}$ on $L$; this is where the factor $|c|$ in \eqref{eq:Tkpartials} is absorbed. Hence, from the formula for $\partial_c F_{k+1}$,
\[ C_{k+1} \leq (1+C_\xi) \cdot \tfrac12 C_\rho + \tfrac12 (1+2C_\xi) e^{2C_6}e^{-2\sigma_0} C_k + C e^{-2\sigma_0},\]
and the formula for $\partial_{\overline{c}} F_{k+1}$ gives the same bound with $C_\rho$ replaced by $C_\xi C_\rho \leq (1+C_\xi)C_\rho$. Once $\sigma_0$ is so large that $(1+2C_\xi)e^{2C_6}e^{-2\sigma_0} \leq 1$, the choice $C := (1+C_\xi)C_\rho + 2Ce^{-2\sigma_0}$ satisfies $C_k \leq C$ for all $k$ by induction. Consequently
\[  |\partial_c F_k| \leq C e^{-2\Re X}, \quad |\partial_{\overline{c}} F_k | \leq Ce^{-2\Re X} \]
for all $k$. These estimates persist as we take the limit $k\to \infty$. As $\partial_c \widetilde{\psi} = (\partial_c \psi) / \psi$, we obtain
\[ \partial_c \psi = O( |z| \cdot |z|^{-2}) = O(|z|^{-1})\]
and similarly for $\partial_{\overline{c}} \psi$, proving~(b). We note that, by \eqref{eq:rho1c}, the constants here do not depend on $c$.

For the part~(a), we will need the following $c$-derivative version of Lemma \ref{lem:Fkdiff}.

\begin{lemma}
\label{lem:Fcdiff}
Given $\alpha\in (1,2)$, for all $k\geq 0$ and all $X \in L$ we have
\[ |\partial_c (F_{k+1} - F_k) | \leq C \max\{1,|c|\} e^{-\alpha^k \Re X},\quad |\partial_{\overline{c}} (F_{k+1} - F_k) | \leq C \max\{1,|c|\} e^{-\alpha^k \Re X} .\]
\end{lemma}

\begin{proof}
The base case $k=0$ follows from \eqref{eq:cTk}, since $F_1 - F_0 = T_1$ and $e^{-2\Re X} \leq e^{-\Re X}$. Now suppose that the claimed inequalities hold for $k-1\geq 0$. Differentiating $F_{k+1} - F_k = \Delta_k + [\xi(W_k) - \xi(W_{k-1})]$ in $c$ and using \eqref{eq:eta} gives
\begin{align*}
\partial_c ( F_{k+1} - F_k) &= (1+\partial_w \xi (W_k) )\partial_c \Delta_k + \partial_{\overline{w}} \xi (W_k) \overline{\partial_{\overline{c}} \Delta_k} \\
& \hskip0.3in + [ \partial_w \xi(W_k) - \partial_w \xi (W_{k-1}) ] \cdot \partial_c W_{k-1} \\
& \hskip0.3in + [\partial_{\overline{w}} \xi (W_k) - \partial_{\overline{w}} \xi(W_{k-1})] \cdot \overline{ \partial_{\overline{c}} W_{k-1} }.
\end{align*}
Now $|\partial_c W_{k-1}|$ and $|\partial_{\overline{c}} W_{k-1}|$ are both bounded by $Ce^{-2\Re X}$, as established above. The factors involving the differences of the partial derivatives of $\xi$ are bounded by $C|\Delta_k|$ by \cite[Lemma 5.4]{FF12}, and $|\Delta_k| \leq C|c|e^{\alpha^{k-1}C_6}e^{-2\alpha^{k-1}\Re X}$ by \eqref{eq:l55}, \eqref{eq:deltak} and \cite[Lemma 5.8]{FF12}. Finally, differentiating \eqref{eq:deltak} in $c$ and using the inductive hypothesis at the point $\widetilde{f}(X)$, together with \eqref{eq:rho1c} and Lemma \ref{lem:Fkdiff}, gives
\[ |\partial_c \Delta_k| + |\partial_{\overline{c}}\Delta_k| \leq C\max\{1,|c|\} e^{\alpha^{k-1}C_6} e^{-2\alpha^{k-1}\Re X}.\]
Putting this together, we have
\begin{align*}
|\partial_c (F_{k+1} - F_k) | & \leq C|\partial_c \Delta_k| + C |\partial_{\overline{c}} \Delta_k| + C|\Delta_k|e^{-2\Re X}\\
&\leq C\max\{1,|c|\} e^{\alpha^{k-1}(C_6 - (2-\alpha)\sigma_0 )} e^{-\alpha^k \Re X},
\end{align*}
exactly as in the proof of Lemma \ref{lem:Fkdiff}, which completes the inductive step. A similar argument runs for the $\overline{c}$-derivatives.
\end{proof}

We have
\[ \sum_{k=0}^{\infty} \sup_{X \in L} | \partial_c (F_{k+1} - F_k ) | \leq C\max\{1,|c|\} \sum_{k=0}^{\infty} e^{-\alpha^k \sigma_0},\]
with the latter sum convergent. It follows that $\partial_c F_k$ converges uniformly on $L$, and similarly for $\partial_XF_k$ by Lemma \ref{lem:Fkdiff}. As $F_k \to F$ uniformly by \cite[Section 5.4]{FF12}, the limit $F$ is $C^1$ and its derivatives are the limits of the derivatives. Each $\partial_c F_k$ is continuous in $(c,X)$ and the convergence is uniform, locally uniformly in $c$, and so the limit derivatives are jointly continuous. This gives~(a) and completes the proof of Proposition \ref{prop:bottreg}.

\section{Important Estimates}
\label{sec:estimates}

Throughout this section and the next, we fix the constant
\begin{equation}
\label{eq:kappa0}
\kappa_0 = \kappa_0(K,\theta) := \max \left \{ 4 , \; 1+\sqrt{2K}, \; 2KC', \; a(K,\theta)^2 + 1 \right \},
\end{equation}
where $C'$ is the constant from Proposition \ref{prop:bottreg}~(d) and $a(K,\theta)$ is the constant from \eqref{eq:A}. Each of the four requirements is used at an identifiable point below, and we indicate them as they arise.

\subsection{Orbit control}

Recall that $P_0(c) = c$ and $P_{n+1}(c) = h(P_n(c))^2 + c$ for $n\geq 0$. We first record a growth estimate for a general orbit of $f_c$. The case of interest is when the orbit starts at $P_m(c)$.

\begin{lemma}
\label{lem:orbit}
Let $K\geq 1$, $c\in \C$ and $\kappa >2$. Suppose that $z\in \C$ satisfies $|z| \geq \max \{ |c|,\kappa \}$ and write $z_n = f_c^n(z)$ and $S = |z|$. Then $(|z_n|)_{n\geq 0}$ is non-decreasing, $|z_n| \geq S$ for all $n\geq 0$, and
\[ |z_n| \geq S(S-1)^{2^n-1}.\]
In particular $z\in I(f_c)$. Moreover, the right hand side is non-decreasing in $S$, so $S$ may be replaced by any lower bound exceeding $1$.
\end{lemma}

\begin{proof}
By \eqref{eq:hsv}, $|h(w)| \geq |w|$ for all $w$, and therefore
\begin{equation}
\label{eq:Pescape1}
|z_{n+1}| = |h(z_n)^2 + c| \geq |h(z_n)|^2 - |c| \geq |z_n|^2- |c|.
\end{equation}
Suppose that $|z_n| \geq S$. As $|c| \leq S \leq |z_n|$, \eqref{eq:Pescape1} gives
\[ |z_{n+1}| \geq |z_n|^2 - |z_n| = |z_n| ( |z_n| -1 ) \geq |z_n| (\kappa-1) \geq |z_n| \geq S,\]
using $\kappa>2$. Hence, by induction, $|z_n| \geq S$ for all $n\geq 0$ and the sequence is non-decreasing.

Next, for every $n$, $|c| \leq S \leq |z_n|$ gives $|c|/|z_n|^2 \leq 1/|z_n| \leq 1/S$, so by \eqref{eq:Pescape1},
\[ |z_{n+1}| \geq |z_n|^2 \left ( 1 - \frac{|c|}{|z_n|^2} \right ) \geq \vartheta |z_n|^2, \quad \vartheta := 1 - \frac{1}{S} \in (0,1).\]
Setting $v_n = \vartheta |z_n|$, we obtain $v_{n+1} \geq \vartheta^2 |z_n|^2 = v_n^2$, and hence $v_n \geq v_0^{2^n}$ for all $n\geq 0$. As $\vartheta S = S-1$ and $\vartheta^{-1} = S/(S-1)$, this gives
\[ |z_n| \geq \frac{(\vartheta S)^{2^n}}{\vartheta} = \frac{S}{S-1}(S-1)^{2^{n}} = S(S-1)^{2^{n}-1}.\]
Finally, $S\mapsto S(S-1)^{2^n-1}$ is non-decreasing on $(1,\infty)$ since $2^n - 1 \geq 0$.
\end{proof}

\begin{corollary}
\label{lem:Pescape}
Let $K\geq 1$, $c \in \C$ and $\kappa>2$. Suppose that $m\geq 0$ and $|P_m(c)| \geq \max \{ |c|,\kappa \}$, and set $S = |P_m(c)|$. Then $c\notin \mathcal{M}_{K,\theta}$, the sequence $(|P_n(c)|)_{n\geq m}$ is non-decreasing, and for each $n\geq m$,
\[ |P_n(c)| \geq S(S-1)^{2^{n-m}-1}.\]
\end{corollary}

\begin{proof}
Apply Lemma \ref{lem:orbit} with $z = P_m(c)$, noting that $P_{m+n}(c) = f_c^n(P_m(c))$ and that $P_n(c) \to \infty$ forces $f_c^n(0) \to \infty$.
\end{proof}

We can now provide the lower bound for the Green's function promised in Section 2.

\begin{lemma}
\label{lem:Glower}
Let $K\geq 1$, $c\in \C$ and $\kappa > 2$, and suppose $|z| \geq \max \{ |c|,\kappa \}$. Then, with $S=|z|$,
\[ \log |z| - \frac{2}{S} \; \leq \; G_c(z) \; \leq \; \log|z| + 2\log K + \frac{1}{S}.\]
Consequently, there is a constant $C_1 = C_1(K,\theta)$ such that
\begin{equation}
\label{eq:Mlower}
G_{\mathcal{M}}(c) \geq \log|c| - C_1 \quad \text{for all } c\in \C \setminus \mathcal{M}_{K,\theta}.
\end{equation}
\end{lemma}

\begin{proof}
Write $z_n = f_c^n(z)$ and $u_n = \log|z_n|$. By Lemma \ref{lem:orbit}, $|z_n| \geq S$ for all $n$ and $|c| \leq |z_n|$, so setting $\eta_n = |c|/|z_n|^2 \leq 1/|z_n| \leq 1/S < 1/2$, we have from \eqref{eq:Pescape1} and $|f_c(z)| \leq K^2|z|^2+|c|$ that
\[ 2u_n + \log ( 1-\eta_n) \leq u_{n+1} \leq 2u_n + 2\log K + \log (1+\eta_n).\]
Summing the telescoping identity $2^{-(n+1)}u_{n+1} - 2^{-n}u_n = 2^{-(n+1)}(u_{n+1}-2u_n)$ and applying Proposition \ref{prop:bottformula} gives
\[ G_c(z) - \log|z| = \sum_{n=0}^{\infty} 2^{-(n+1)} ( u_{n+1} - 2u_n).\]
As $(\eta_n)$ is non-increasing with $\eta_0 \leq 1/S$, we have $\sum_n 2^{-(n+1)}\eta_n \leq \eta_0 \leq 1/S$, and the bounds follow using $\log(1+x) \leq x$ and $\log(1-x) \geq -2x$ for $0\leq x \leq 1/2$.

For \eqref{eq:Mlower}, take $z=c$. If $|c| \geq \kappa_0$ then the hypothesis holds with $\kappa=\kappa_0$, recalling \eqref{eq:kappa0}, and $S=|c| \geq 4$, so $G_{\mathcal{M}}(c) \geq \log|c| - 1/2$. If $|c|<\kappa_0$ then $\log|c| - \log \kappa_0 < 0 \leq G_{\mathcal{M}}(c)$. So $C_1 = \log \kappa_0$ suffices, as $\log \kappa_0 \geq \log 4 > 1/2$.
\end{proof}

The Green's function of \cite{BF25} is continuous in $z$ for each fixed $c$. We will need continuity jointly in $(c,z)$, which we deduce from Proposition \ref{prop:bottformula} and the two-sided bounds above.

\begin{lemma}
\label{lem:Gcont}
The function $(c,z) \mapsto G_c(z)$ is continuous on $\C \times \C$. In particular, $G_{\mathcal{M}}$ is continuous on $\C$, the sets $\{G_{\mathcal{M}} > t\}$ are open and the sets $\{G_{\mathcal{M}} \geq t \}$ are closed.
\end{lemma}

\begin{proof}
Applying Corollary \ref{cor:bott} at the point $f_c^n(z)$ and using $G_c(f_c^n(z)) = 2^nG_c(z)$, we obtain, for every $n\geq 0$,
\[ G_c(z) \leq 2^{-n} \left ( \log^+ |f_c^n(z)| + \log^+|c| + \log (K^2+1) \right ) =: u_n(c,z).\]
Each $u_n$ is continuous on $\C\times \C$, so $G = \inf_n u_n$ is upper semicontinuous. If $z\in BO(f_c)$, then $(f_c^n(z))_n$ is bounded, so $u_n(c,z) \to 0 = G_c(z)$; given $\epsilon>0$ choose $n$ with $u_n(c,z)<\epsilon$, and then $0 \leq G_{c'}(z') \leq u_n(c',z') <\epsilon$ for $(c',z')$ near $(c,z)$. Hence $G$ is continuous at every point $(c,z)$ with $z \in BO(f_c)$.

If instead $z \in I(f_c)$, choose $N$ with $|f_c^N(z)| > \max \{ |c| , 3\}$. This is an open condition, so there is a neighbourhood $V$ of $(c,z)$ on which $|f_{c'}^N(z')| \geq \max\{|c'|,3\}$, and then Lemma \ref{lem:Glower} with $\kappa=3$, applied to $w = f_{c'}^N(z')$, gives $|G_{c'}(w) - \log|w|| \leq 2\log K + 1$. Dividing by $2^N$,
\[ \left | G_{c'}(z') - 2^{-N} \log |f_{c'}^N(z')| \right | \leq 2^{-N}(2\log K+1) \quad \text{on } V.\]
As $(c',z') \mapsto 2^{-N}\log|f_{c'}^N(z')|$ is continuous on $V$, and $N$ may be taken as large as we please, $G$ is continuous at $(c,z)$.
\end{proof}

Recall $t_0$ and $t_*(r)$ from \eqref{eq:t0} and \eqref{eq:tstar} respectively.

\begin{lemma}
\label{lem:tstar}
For every $r\geq 2$ we have $t_*(r) \leq \log r + \log (K^2+1) < \infty$. Consequently $t_0 \leq \log 2 + \log(K^2+1)$.
\end{lemma}

\begin{proof}
By Corollary \ref{cor:bott}, $G_{\mathcal{M}}(c) = G_c(c) \leq \log \max\{|c|,1\} + \log(K^2+1)$, so $\{ G_{\mathcal{M}} > t \} \subset \{ |c| > e^{t - \log(K^2+1)} \}$ whenever $t \geq \log(K^2+1)$. If $e^{t-\log(K^2+1)} \geq r$, then any $c$ with $G_{\mathcal{M}}(c)>t$ satisfies $|c| > r \geq 2$, and then \eqref{eq:Pescape1} gives inductively $|P_{i+1}(c)| \geq |P_i(c)|^2 - |c| \geq |P_i(c)|(|P_i(c)|-1) \geq |P_i(c)|$ whenever $|P_i(c)| \geq \max\{|c|,2\}$, so $|P_i(c)| \geq |c| > r$ for all $i$ and $c\in \Omega_r$. (The restriction $r \geq 2$ is needed at this step: for $K=1$, the parameter $c = -1.4+0.5i$ escapes and has $|c| \approx 1.487$, but $|P_1(c)| \approx 0.952$.) For the last assertion, if $t>\log 2 + \log(K^2+1)$ then $r:= e^{t-\log(K^2+1)}>2$ and the above shows $\{G_{\mathcal{M}}>s\} \subset \Omega_r \subset \Omega^*$ for all $s\geq t$, so $t_0 \leq t$.
\end{proof}

\subsection{The terms that control the orientation of $P_n$}
\label{sec:orientation}

Differentiating the formulas $P_0(c)=c$ and $P_{n+1}(c) = h(P_n(c))^2 + c$ gives $D_cP_0 = I$ and
\begin{equation}
\label{eq:DcPrec}
D_c P_{n+1} = (A_nh)D_cP_n + I,
\end{equation}
where $A_n$ denotes multiplication by $\alpha_n := 2h(P_n(c))$.

If we set $\Pi_0 = I$ and, for $j\geq 1$,
\[ \Pi_j = (A_{j-1}h)\ldots (A_0h),\]
so that the factors are ordered with decreasing index from left to right, then $\Pi_j$ is invertible provided $P_i(c) \neq 0$ for $i<j$. We assume for the rest of this subsection that $P_i(c) \neq 0$ for all $i\geq 0$. This holds for every $c\notin \mathcal{M}_{K,\theta}$, since $P_i(c) = 0$ would make $0$ periodic under $f_c$. Iterating \eqref{eq:DcPrec} and using $\prod_{i=j}^{n-1}(A_ih) = \Pi_n \Pi_j^{-1}$ gives
\begin{equation}
\label{eq:DcPn}
D_cP_n = \Pi_nR_n, \quad R_n = I +\sum_{j=1}^n \Pi_j^{-1}.
\end{equation}

\begin{lemma}
\label{lem:alg1}
We have $\det \Pi_n > 0$ for every $n$, and so
\[ \sign \det D_c P_n = \sign \det R_n.\]
\end{lemma}

\begin{proof}
This follows directly from $\det (A_jh) = |\alpha_j|^2 \det h = 4|h(P_j(c))|^2K >0$ and \eqref{eq:DcPn}.
\end{proof}

The aim of the rest of this subsection is to show that if we have some control on the orbit $P_n(c)$, then this implies $R_n$ is close to the identity. 

\begin{lemma}
\label{lem:alg2}
For every $c$ with $P_i(c) \neq 0$ for all $i$, and every $j\geq 1$,
\begin{equation}
\label{eq:Piinv}
|| \Pi_j^{-1} || = \ell(\Pi_j)^{-1} \leq \prod_{i=0}^{j-1} \frac{1}{2|h(P_i(c))|} \leq \prod_{i=0}^{j-1} \frac{1}{2|P_i(c)|},
\end{equation}
so that $||R_n - I|| \leq \sigma_n(c)$ for every $n$, with $\sigma_n$ as in Definition \ref{def:sigma}. If $c\notin \mathcal{M}_{K,\theta}$, then $\sigma(c) <\infty$ and the sequence $(R_n)_{n=0}^{\infty}$ converges to a limit $R_{\infty}$ with $||R_\infty - I|| \leq \sigma(c)$.
\end{lemma}

\begin{proof}
Each factor of $\Pi_j$ satisfies $\ell(A_ih) = |\alpha_i|\,\ell(h) = 2|h(P_i(c))|$ by \eqref{eq:hsv}, since $A_i$ is a similarity, and $\ell$ is supermultiplicative; this gives \eqref{eq:Piinv}, the last inequality being $|h(w)| \geq |w|$. Summing over $j$ gives $||R_n-I|| \leq \sigma_n(c)$.
If $c\notin \mathcal{M}_{K,\theta}$, then $|P_n(c)| \to \infty$, so there is $N$ with $|P_i(c)| \geq 2$ for all $i \geq N$. Hence for $j>N$ the products in \eqref{eq:Piinv} decay at least geometrically with ratio $1/4$, so $\sigma(c)<\infty$ and the series defining $R_n$ converges absolutely.
\end{proof}

\begin{lemma}
\label{lem:Rnclose}
Suppose that $c\in \C$ and $n\geq 1$ are such that $\sigma_n(c) <1$. Then the matrix $R_n$ satisfies
\[ ||R_n - I || \leq \sigma_n(c) <1,\]
so $R_n$ is invertible, $\det R_n \geq (1-\sigma_n(c))^2 >0$, and the linear dilatation satisfies
\[ \mathcal{H}(R_n) \leq \frac{ 1 + \sigma_n(c)}{1-\sigma_n(c)}.\]
If $c\notin \mathcal{M}_{K,\theta}$ and $\sigma(c)<1$, the same bounds hold for $R_{\infty} = \lim R_n$ with $\sigma(c)$ in place of $\sigma_n(c)$. Consequently, $\det D_cP_n > 0$ for all $n\geq 1$ whenever $\sigma(c)<1$.

In particular, if $\kappa := \inf_{i\geq 0} |P_i(c)| >1$, then $\sigma(c) \leq (2\kappa-1)^{-1} <1$, and the conclusions hold with
\[ ||R_n - I || \leq \frac{1}{2\kappa - 1},\qquad \mathcal{H}(R_n) \leq \frac{ 1 + (2\kappa-1)^{-1}}{1-(2\kappa-1)^{-1}} = \frac{\kappa}{\kappa-1}.\]
\end{lemma}

\begin{proof}
The bound $||R_n - I|| \leq \sigma_n(c)$ is Lemma \ref{lem:alg2}. Writing $\epsilon = \sigma_n(c) <1$, Weyl's inequality implies both singular values of $R_n$ lie in $[1-\epsilon, 1+\epsilon]$. This gives invertibility, $|\det R_n| \geq (1-\epsilon)^2$ and the claimed linear dilatation. Considering the path $t\mapsto I +t(R_n-I)$ for $t\in [0,1]$, we have $\ell(I+t(R_n-I)) \geq 1-\epsilon>0$, so $\det$ cannot vanish along the path and hence $\det R_n$ has the same sign as $\det I$. The statements for $R_\infty$ follow in the same way from $||R_\infty - I|| \leq \sigma(c)$, and $\sigma_n \leq \sigma$ gives $\det D_cP_n = \det \Pi_n \cdot \det R_n>0$ for every $n$, using Lemma \ref{lem:alg1}. Finally, if $|P_i(c)| \geq \kappa >1$ for all $i$, then \eqref{eq:Piinv} gives $||\Pi_j^{-1}|| \leq (2\kappa)^{-j}$ and $\sigma(c) \leq \sum_{j\geq 1}(2\kappa)^{-j} = (2\kappa-1)^{-1}$.
\end{proof}

\subsection{Gap between the parameter and dynamical sides}
\label{sec:gap}

Let $c \notin \mathcal{M}_{K,\theta}$. By Theorem \ref{thm:bottcher} and the functional equation for the extended B\"ottcher coordinate (see Section 1), we have $f_0^n \circ \psi_c = \psi_c \circ f_c^n$ on the extended domain of $\psi_c$, and evaluating at $z=c$ gives
\[ f_0^n(\Psi(c)) = \psi_c(P_n(c)) =: Q_n(c) \]
for every $n$ such that $P_n(c)$ lies in the domain of $\psi_c$, in particular for every $n$ with $(c,P_n(c)) \in \mathcal{D}$, in which case $\psi_c$ is the unextended coordinate. We show in Proposition \ref{prop:derivative} below that $\Psi$ is $C^1$ and that
\begin{equation}
\label{eq:DCPsi}
D_c\Psi = [ Df_0^n (\Psi(c)) ]^{-1} D_c Q_n
\end{equation}
for all large $n$. In this subsection we give the estimates needed to analyse the right hand side.

On the parameter side, $\Pi_n = \prod_{j=0}^{n-1} (A_j h)$, where $A_j$ is multiplication by $\alpha_j = 2h(P_j(c))$. If we set $\phi_j = \arg \alpha_j$ and denote by $\Rot_{\phi}$ the rotation by angle $\phi$, then since multiplication by $\alpha_j$ equals $|\alpha_j|\Rot_{\phi_j}$ and we can move positive scalars around, we may write
\[ \Pi_n = \left ( \prod_{j=0}^{n-1} |\alpha_j| \right ) Y_n,\quad Y_n:= \prod_{j=0}^{n-1} ( \Rot_{\phi_j} h) .\]
This term appears within the $D_cQ_n$ term in \eqref{eq:DCPsi}.

On the dynamical side, we have $Df_0^n (\Psi(c)) = \prod_{j=0}^{n-1} (B_jh)$ where $B_j$ is multiplication by $\beta_j = 2h(w_j)$, with $w_j = f_0^j(\Psi(c))$. Note that $w_j$ is defined for all $j\geq 0$ as the $f_0$-orbit of $\Psi(c)$, and that $w_j = \psi_c(P_j(c))$ wherever $\psi_c$ is defined; moreover $w_j \neq 0$, since $\Psi(c) \in I(f_0)$ and $0\in BO(f_0)$. Setting $\chi_j = \arg \beta_j$, we likewise have
\[ Df_0^n (\Psi(c)) = \left ( \prod_{j=0}^{n-1} |\beta_j| \right ) \tilde Y_n,\quad \tilde Y_n := \prod_{j=0}^{n-1} (\Rot_{\chi_j} h).\]
This term will appear when we compute the $[ Df_0^n (\Psi(c)) ]^{-1}$ term in \eqref{eq:DCPsi}. The point is that $Y_n$ and $\tilde Y_n$ are identical in form and differ only in the rotation angles. We therefore set
\[ \delta_j = \phi_j - \chi_j, \quad \epsilon_j = |\delta_j|,\]
where $\delta_j$ is taken to be the principal value, so that we always have $\epsilon_j \leq \pi$. The following lemma gives the estimate we use once the orbit is large.

\begin{lemma}
\label{lem:epsilon}
Suppose $|P_j(c)| \geq \max \{ |c| , \kappa_0 \}$, recalling $\kappa_0$ from \eqref{eq:kappa0}. Then $|w_j| \geq |P_j(c)|/2$ and
\[ \epsilon_j \leq \frac{2C'K|c|}{|P_j(c)|^2},\]
where $C'$ is the constant from Proposition \ref{prop:bottreg}~(d).
\end{lemma}

\begin{proof}
We first check that $(c,P_j(c)) \in \mathcal{D}$, so that $w_j = \psi_c(P_j(c))$ with $\psi_c$ its original, unextended, B\"ottcher coordinate, and Proposition \ref{prop:bottreg} applies at $P_j(c)$: by the fourth requirement in \eqref{eq:kappa0}, $|P_j(c)|^2 \geq (a^2+1)|P_j(c)| \geq a^2|c| + a^2 = a^2(1+|c|)$. Next, since $|P_j(c)| \geq |c|$ we have $|P_j(c)|^2 \geq \kappa_0 |c| \geq 2KC'|c|$, using the third requirement in \eqref{eq:kappa0}. Hence 
\begin{equation}
\label{eq:epsilon1}
C'|c|/|P_j(c)| \leq |P_j(c)|/(2K) \leq |P_j(c)|/2,
\end{equation}
and Proposition \ref{prop:bottreg}~(d) gives
\[ |w_j| \geq |P_j(c)| - C'|c|/|P_j(c)| \geq |P_j(c)|/2 .\]
Next, by \eqref{eq:hsv} and \eqref{eq:epsilon1},
\[ |h(w_j) - h(P_j(c))| \leq K |w_j - P_j(c)| \leq \frac{ KC'|c|}{|P_j(c)|} \leq \frac{|P_j(c)|}{2} \leq \frac{|h(P_j(c))|}{2}.\]
Writing $h(w_j)/h(P_j(c)) = 1+\omega$, we therefore have $|\omega| \leq 1/2$, and so
\[ \epsilon_j = |\arg (1+\omega) | \leq |\log (1+\omega)| \leq 2|\omega| \leq \frac{2C'K|c|}{|P_j(c)|^2},\]
as claimed.
\end{proof}

We now set $S_n = \tilde Y_n^{-1} Y_n$, and, writing $g_j = \Rot_{\phi_j}h$ and $g_j' = \Rot_{\chi_j}h$ so that $Y_n = g_{n-1}\ldots g_0$ and $\tilde Y_n = g_{n-1}'\ldots g_0'$,
\[ v_j := (g_j')^{-1}g_j = h^{-1} \Rot_{\delta_j} h, \qquad \mathcal{E}_j:= Y_j^{-1}(v_j-I)Y_j .\]
Since $S_n = \tilde Y_{n-1}^{-1} v_{n-1} Y_{n-1}$, we have
\[ S_n - S_{n-1} = \tilde{Y}_{n-1}^{-1} ( v_{n-1} - I) Y_{n-1} = S_{n-1} \mathcal{E}_{n-1},\]
and so $S_n = S_{n-1} (I+\mathcal{E}_{n-1})$. As $S_0 = I$, this leads to the formula
\[ S_n = \prod_{j=0}^{n-1} (I+\mathcal{E}_j),\]
where here the factors are ordered with increasing index from left to right. If we now set
\begin{equation}
\label{eq:s}
s:= \sum_{j=0}^{\infty} ||\mathcal{E}_j||,
\end{equation}
we have the following result. We emphasize that no smallness of $s$ is assumed. This is essential for Theorem \ref{thm:potential}, where the individual $||\mathcal{E}_j||$ may exceed $1$.

\begin{lemma}
\label{lem:Snestimate}
With the notation above, we have
\[ ||\mathcal{E}_j|| \leq K^{j+1}\epsilon_j \]
for every $j\geq 0$, and for all $n\in \N$,
\begin{equation}
\label{eq:Sn1} 
\mathcal{H}(S_n) \leq e^{2s}, \qquad ||S_n - I|| \leq \prod_{j=0}^{n-1} ( 1 + ||\mathcal{E}_j ||) - 1 \leq e^s-1.
\end{equation}
\end{lemma}

\begin{proof}
We first prove the bound on $||\mathcal{E}_j||$. By \eqref{eq:hsv} we have $||h^{-1}||=1$ and $||h||=K$, and $||\Rot_{\delta}-I|| = 2|\sin (\delta/2)| \leq |\delta|$, so $||v_j - I|| \leq K\epsilon_j$. Each factor $\Rot_{\phi_i}h$ of $Y_j$ has $\mathcal{H}(\Rot_{\phi_i}h) = \mathcal{H}(h) = K$, so $\mathcal{H}(Y_j) \leq K^j$ by submultiplicativity, and therefore
\[ ||\mathcal{E}_j||\leq ||Y_j^{-1}|| \cdot ||v_j-I|| \cdot ||Y_j|| = \mathcal{H} (Y_j) \, ||v_j - I|| \leq K^j \cdot K\epsilon_j = K^{j+1}\epsilon_j.\]

The second estimate in \eqref{eq:Sn1} is immediate from $S_n = \prod_{j=0}^{n-1}(I+\mathcal{E}_j)$ and $1+x \leq e^x$. For the first, observe that $I + \mathcal{E}_j = Y_j^{-1}v_jY_j$, so that
\[ S_n = \prod_{j=0}^{n-1} Y_j^{-1}v_jY_j \]
and submultiplicativity of $\mathcal{H}$ gives $\mathcal{H}(S_n) \leq \prod_{j=0}^{n-1} \mathcal{H}(Y_j^{-1}v_jY_j)$. Now $$\det v_j = \det (h)^{-1} \det ( \Rot_{\delta_j}) \det (h) =1,$$ so $\det (Y_j^{-1} v_j Y_j) = 1$. For a real $2\times 2$ matrix $A$ with $\det A = 1$ we have $||A||\ell(A)=1$, and hence $\mathcal{H} (A) = ||A||^2$. Therefore
\[ \log \mathcal{H} (Y_j^{-1}v_jY_j) = 2\log ||I+\mathcal{E}_j || \leq 2\log (1+||\mathcal{E}_j||) \leq 2||\mathcal{E}_j||,\]
with no restriction on the size of $||\mathcal{E}_j||$. Summing over $j$ gives $\log \mathcal{H}(S_n) \leq 2s$, as claimed. We remark that the elementary bound $||S_n - I|| \leq e^s-1$ would only give $\mathcal{H}(S_n) \leq e^s/(2-e^s)$, and only for $s<\log 2$; it is the identity $\det v_j = 1$ that yields a bound without any smallness assumption on $s$.
\end{proof}

\subsection{The derivative of $\Psi$}
\label{sec:derivative}

We now assemble the estimates of this section into a formula for $D_c\Psi$ that is valid at every parameter outside $\mathcal{M}_{K,\theta}$. For $c\in \C\setminus \mathcal{M}_{K,\theta}$ set $\kappa_c := \max \{ |c|,\kappa_0 \}$ and define the \emph{escape time}
\begin{equation}
\label{eq:escape}
m(c) := \min \{ i\geq 0 : |P_i(c)| \geq \kappa_c \},
\end{equation}
which is finite since $P_i(c) \to \infty$, and which is $0$ when $|c| \geq \kappa_0$. Writing $m = m(c)$ and $S := |P_m(c)| \geq \kappa_c$, Corollary \ref{lem:Pescape} with $\kappa = \kappa_0$ shows that $(|P_n(c)|)_{n\geq m}$ is non-decreasing and
\begin{equation}
\label{eq:thmp1}
|P_n(c)| \geq S(S-1)^{2^{n-m}-1} \geq \kappa_0 (\kappa_0-1)^{2^{n-m}-1} \geq 3^{2^{n-m}} \qquad (n \geq m),
\end{equation}
so the orbit grows double exponentially from time $m$ onwards and never returns below $\kappa_c$. Moreover, since $\kappa_c \geq \max\{|c|, a^2+1\}$ we have $\kappa_c^2 \geq (a^2+1)\kappa_c = a^2\kappa_c + \kappa_c \geq a^2|c| + a^2$, so that $\kappa_c \geq a\sqrt{1+|c|}$ and
\begin{equation}
\label{eq:inD}
(c, P_n(c)) \in \mathcal{D} \qquad \text{for all } n \geq m(c);
\end{equation}
in particular Lemma \ref{lem:epsilon} applies at $P_n(c)$ for every $n \geq m(c)$.

We first address the continuity of $\Psi$. This is not immediate from Definition \ref{def:Psi}, because the extension of $\psi_c$ to a domain containing $c$ is constructed in \cite{FF12} separately for each $c$.

\begin{lemma}
\label{lem:Psicont}
$\Psi$ is continuous on $\C\setminus \mathcal{M}_{K,\theta}$.
\end{lemma}

\begin{proof}
Fix $c_0 \in \C\setminus \mathcal{M}_{K,\theta}$ and a closed disc $N$ centred at $c_0$ and contained in $\C\setminus\mathcal{M}_{K,\theta}$. The disk $N$ will be shrunk as needed.

We begin by describing the extended domain of $\psi_c$ in terms of $G_c$. By Corollary \ref{cor:bott} there is $T = T(N)$ such that $V_c := \{ G_c > T \} \subset \{ |z| \geq a\sqrt{1+|c|} \}$ for every $c\in N$. Thus $\psi_c$ is defined and injective on $V_c$, satisfies $f_0\circ \psi_c = \psi_c \circ f_c$ there, and $G_c = G_0 \circ \psi_c$ on $V_c$ by \cite[Definition 3.6]{BF25}. 

The image of $\psi_c$ contains a neighbourhood of infinity: for $\varrho \geq a\sqrt{1+|c|}$, Proposition \ref{prop:bottreg}~(d) gives $|\psi_c(z) - z| \leq C'|c|/\varrho$ on the circle $\{|z| = \varrho\}$, so the winding number of $\psi_c(\{|z|=\varrho\})$ about any $w$ with $|w| < \varrho - C'|c|/\varrho$ is $1$, and every such $w$ is in the image. As $N$ is compact, enlarging $T$ we may therefore assume that $\{ G_0 > T\}$ lies in the image of $\psi_c$ for every $c\in N$, so that $\psi_c$ is a homeomorphism from $V_c$ onto $\{G_0>T\}$; moreover $\psi_c(z) \to \infty$ as $z\to\infty$, so $\psi_c$ extends to a homeomorphism from $V_c \cup \{\infty\}$ onto $\{G_0>T\}\cup\{\infty\}$, which is simply connected by Lemma \ref{lem:jordan}. Hence $V_c$ is connected with connected complement. 

Since $G_{\mathcal{M}}$ is continuous, by Lemma \ref{lem:Gcont}, and positive on $N$, we may also choose $T$ so that $T > 2G_{\mathcal{M}}(c_0)$ and $G_{\mathcal{M}}(c_0)$ is not of the form $T2^{-k}$, and then shrink $N$ so that there is a single integer $n \geq 1$ with
\[ s := T2^{-n} < G_{\mathcal{M}}(c) \leq 2s \qquad \text{for all } c\in N.\]
For $0 \leq k<n$, the critical value $c$ of $f_c$ does not lie in $\{G_c > T2^{-k}\}$, and so, as $f_c^{-1}( \{G_c > T2^{-k}\}) = \{G_c > T2^{-(k+1)}\}$ by the functional equation $G_c \circ f_c = 2G_c$, the map $f_c : \{G_c > T2^{-(k+1)}\} \to \{G_c > T2^{-k}\}$ is a two-to-one covering map. Applying \cite[Lemma 6.1]{FF12} $n$ times, exactly as in the proof of \cite[Theorem 2.4(ii)]{FF12} but with $V_c$ in place of the neighbourhood of infinity used there, extends $\psi_c$ to
\[ D_c := \{ G_c > s \} \ni c,\]
and $\Psi(c) = \psi_c(c)$ is the value of this extension. 

All we shall use about the extension is that it is continuous on $D_c$ and satisfies $f_0^n \circ \psi_c = \psi_c \circ f_c^n$ there (the $n$-fold iterate of the conjugacy, which makes sense as $f_c^k(D_c) \subset \{G_c > 2^k s\} \subset D_c$). This characterizes it by path lifting, as follows. The map $f_0(z) = h(z)^2$ vanishes only at $z=0$, so $f_0^n : \C\setminus \{0\} \to \C\setminus \{0\}$ is a covering map of degree $2^n$. For $z\in D_c$, let $\gamma$ be a curve in $D_c$ from a point $z_1 \in V_c$ to $z$, recalling that $D_c$ is connected. Then $f_c^n(\gamma)$ is a curve in $V_c$, so $\psi_c \circ f_c^n \circ \gamma$ is a curve in $\{G_0 >T\} \subset \C\setminus\{0\}$, and $\psi_c\circ \gamma$ is a lift of it under $f_0^n$ with initial point $\psi_c(z_1)$. By uniqueness of path lifting, $\psi_c\circ\gamma$ is \emph{the} lift with this initial point, and in particular $\psi_c(z)$ is its endpoint.

Now let $c_1 \in N$; we show that $\Psi$ is continuous at $c_1$. By Lemma \ref{lem:Glower}, any $z_\infty$ with $|z_\infty| \geq \max\{ \sup_N |c|, 4, e^{T+1}\}$ satisfies $G_c(z_\infty)>T$ for every $c\in N$; fix such a point. Choose a curve $\beta$ in $D_{c_1}$ from $z_\infty$ to $c_1$. Since $\beta$ is compact, $G_{c_1}>s$ on $\beta$, and $G$ is jointly continuous, there is a neighbourhood $N_1 \subset N$ of $c_1$ such that, for every $c \in N_1$, both $\beta$ and the segment $[c_1,c]$ lie in $D_c$. For $c\in N_1$ let $\gamma_c$ be the curve $\beta$ followed by the segment $[c_1,c]$, parametrised by $[0,1]$; it lies in $D_c$, starts at $z_\infty \in V_c$ and ends at $c$. The map
\[ H : N_1 \times [0,1] \to \C\setminus \{0\}, \qquad H(c,\tau) := \psi_c \left ( f_c^n ( \gamma_c(\tau) ) \right ),\]
is continuous, because $f_c^n(\gamma_c(\tau))$ is continuous in $(c,\tau)$ and lies in $V_c \subset \{|z| \geq a\sqrt{1+|c|}\}$, where $\psi$ is jointly continuous by Proposition \ref{prop:bottreg}~(a). Moreover, $c \mapsto \psi_c(z_\infty)$ is continuous for the same reason. By the homotopy lifting property of the covering map $f_0^n$, there is a unique continuous $\widetilde{H} : N_1 \times [0,1] \to \C \setminus \{0\}$ with $f_0^n \circ \widetilde{H} = H$ and $\widetilde{H}(c,0) = \psi_c(z_\infty)$. For each fixed $c$, $\widetilde{H}(c,\cdot)$ is the lift of $\psi_c\circ f_c^n \circ \gamma_c$ with initial point $\psi_c(z_\infty)$, so $\widetilde{H}(c,1) = \psi_c(c) = \Psi(c)$ by the characterisation above. Hence $\Psi = \widetilde{H}(\cdot,1)$ is continuous on $N_1$.
\end{proof}

Next we establish a crucial formula for evaluating the derivative of $\Psi$.

\begin{proposition}
\label{prop:derivative}
Let $c\in \C\setminus \mathcal{M}_{K,\theta}$ and let $m = m(c)$. Then $\Psi$ is $C^1$ in a neighbourhood of $c$, and \eqref{eq:DCPsi} holds for every $n \geq m+1$. Moreover, in the notation of Sections \ref{sec:orientation} and \ref{sec:gap}, the limits $R_\infty = \lim_n R_n$, $S_\infty = \lim_n S_n$ and $\lambda_\infty = \lim_n \lambda_n$, where
\[ \lambda_n := \prod_{j=0}^{n-1} \frac{|h(P_j(c))|}{|h(w_j)|},\]
all exist, with $\lambda_\infty \in (0,\infty)$ and $\det S_\infty = 1$, and
\begin{equation}
\label{eq:formula}
D_c\Psi = \lambda_\infty S_\infty R_\infty .
\end{equation}
In particular $\sign \det D_c\Psi = \sign \det R_\infty$, and $c \mapsto \det D_c \Psi$ is continuous on $\C \setminus \mathcal{M}_{K,\theta}$. Finally, the quantity $s$ of \eqref{eq:s} satisfies
\begin{equation}
\label{eq:sgeneral}
s \leq \pi m K^m + \frac{4C'K^2|c| K^m}{(|P_m(c)|-1)^2}.
\end{equation}
\end{proposition}

\begin{proof}
Write $S = |P_m(c)|$ and fix $n \geq m+1$. By \eqref{eq:thmp1} and the monotonicity of $(|P_n(c)|)_{n\geq m}$, 
\[ |P_n(c)| \geq |P_{m+1}(c)| \geq S(S-1) \geq 3S \geq 3\kappa_c > a\sqrt{1+|c|},\] 
so $(c,P_n(c))$ lies in the interior of $\mathcal{D}$, and hence so does $(c',P_n(c'))$ for $c'$ in a neighbourhood $N$ of $c$. On $N$, the map $c'\mapsto Q_n(c') = \psi_{c'}(P_n(c'))$ is $C^1$ by Proposition \ref{prop:bottreg}~(a), since $c'\mapsto P_n(c')$ is polynomial in $\Re c', \Im c'$, and $f_0^n(\Psi(c')) = Q_n(c')$ holds on $N$ by the functional equation. The map $f_0^n$ is a local diffeomorphism at $\Psi(c)$: $f_0$ is a local diffeomorphism away from $0$, and the $f_0$-orbit of $\Psi(c) \in I(f_0)$ never meets $0 \in BO(f_0)$. Let $\Theta$ be a smooth local inverse of $f_0^n$ near $Q_n(c)$ with $\Theta(Q_n(c)) = \Psi(c)$. As $\Psi$ is continuous by Lemma \ref{lem:Psicont}, $\Psi = \Theta \circ Q_n$ on a neighbourhood of $c$, so $\Psi$ is $C^1$ there, and the chain rule gives \eqref{eq:DCPsi}.

Next, the real form of the chain rule gives $D_cQ_n = L_nD_cP_n + J_n$, where, by \eqref{eq:inD} and Proposition \ref{prop:bottreg}~(b),(c),
\[ L_n = D_z\psi_c|_{z=P_n(c)} = I + O(|c|/|P_n(c)|^2), \qquad ||J_n|| = O(|P_n(c)|^{-1}),\]
and $D_cP_n = \Pi_nR_n$ by \eqref{eq:DcPn}. On the dynamical side, $Df_0^n(\Psi(c)) = \prod_{j=n-1}^{0}(B_jh)$ in the notation of Section \ref{sec:gap}, so that, since $\ell(h) = 1$, $\ell$ is supermultiplicative, $w_j \neq 0$, and $|h(w_j)| \geq |w_j| \geq |P_j(c)|/2$ for $j\geq m$ by Lemma \ref{lem:epsilon},
\begin{equation}
\label{eq:ellbound}
\ell(Df_0^n(\Psi(c))) \geq \prod_{j=0}^{n-1} 2|h(w_j)| \geq \Lambda_m \prod_{j=m}^{n-1} |P_j(c)|, \qquad \Lambda_m := \prod_{j=0}^{m-1} 2|w_j| >0.
\end{equation}
Hence
\[ \left|\left| [Df_0^n(\Psi(c))]^{-1} J_n \right|\right| \leq \frac{C}{\Lambda_m} \prod_{j=m}^{n-1} |P_j(c)|^{-1} \cdot |P_n(c)|^{-1} \to 0\]
as $n\to \infty$, by \eqref{eq:thmp1}. Similarly, $||\Pi_n|| \leq \prod_{j<n} 2K^2|P_j(c)|$ and 
\[ ||R_n|| \leq 1 + \sigma_n(c) \leq 1 + \sigma(c) < \infty\] 
by \eqref{eq:Piinv} and Lemma \ref{lem:alg2}, so
\[ \left|\left| [Df_0^n(\Psi(c))]^{-1} (L_n - I)\Pi_nR_n \right|\right| \leq \frac{ \prod_{j<m} 2K^2|P_j(c)| }{\Lambda_m} \cdot (2K^2)^{n-m} \cdot \frac{ C|c|(1+\sigma(c))}{|P_n(c)|^2} \to 0,\]
since $|P_n(c)| \geq 3^{2^{n-m}}$ dominates $(2K^2)^{n-m}$. Therefore, writing 
\[ [Df_0^n(\Psi(c))]^{-1}\Pi_n = (\prod_{j<n}|\beta_j|)^{-1} \tilde Y_n^{-1} (\prod_{j<n}|\alpha_j|)Y_n = \lambda_n S_n,\]
we have
\[ D_c\Psi = \lim_{n\to \infty} [Df_0^n(\Psi(c))]^{-1} \Pi_n R_n = \lim_{n\to\infty} \lambda_n S_n R_n .\]

It remains to show that the three factors converge. First, $R_n \to R_\infty$ by Lemma \ref{lem:alg2}. Second, for $j \geq m$, \eqref{eq:hsv} and Proposition \ref{prop:bottreg}~(d) give
\[ \left | \, |h(P_j(c))| - |h(w_j)| \, \right | \leq K|P_j(c) - w_j| \leq \frac{KC'|c|}{|P_j(c)|},\]
and $|h(w_j)| \geq |P_j(c)|/2$, so
\begin{equation}
\label{eq:lambda}
\frac{ |h(P_j(c))|}{|h(w_j)|} = 1 + O\left ( \frac{ K|c| }{|P_j(c)|^2}\right ), \qquad j \geq m,
\end{equation}
with the errors absolutely summable by \eqref{eq:thmp1}. Moreover, for $j\geq m$ the ratio in \eqref{eq:lambda} lies in $[2/3, 2]$, since $KC'|c|/|P_j(c)|^2 \leq \tfrac12$ by the third requirement in \eqref{eq:kappa0}, so its logarithm is bounded by a constant multiple of the error term. As every factor of $\lambda_n$ is positive, and the first $m$ factors are fixed, $\lambda_n$ converges to some $\lambda_\infty \in (0,\infty)$. Third, we prove \eqref{eq:sgeneral}. Since $\epsilon_j \leq \pi$ always, Lemma \ref{lem:Snestimate} gives, for the pre-escape block,
\[ \sum_{j=0}^{m-1} ||\mathcal{E}_j|| \leq \pi \sum_{j=0}^{m-1} K^{j+1} \leq \pi mK^m,\]
using $\sum_{j<m}K^j \leq mK^{m-1}$, which does not degenerate as $K\to 1$. For the post-escape block, Lemma \ref{lem:epsilon} applies for $j\geq m$, and \eqref{eq:thmp1} gives $|P_j(c)|^2 \geq \Lambda^{2^{j-m+1}}$ with $\Lambda := S-1 \geq \kappa_0 - 1$, so
\[ \sum_{j=m}^{\infty} ||\mathcal{E}_j|| \leq \sum_{j=m}^{\infty} K^{j+1} \frac{2C'K|c|}{|P_j(c)|^2} \leq 2C'K^2|c|K^m \sum_{i=0}^{\infty} K^i \Lambda^{-2^{i+1}}.\]
The ratio of consecutive terms in the last sum is $K\Lambda^{-2^{i+1}} \leq K\Lambda^{-2} \leq \tfrac12$, because $\Lambda \geq \kappa_0 - 1 \geq \sqrt{2K}$ by the second requirement in \eqref{eq:kappa0}. The sum is therefore dominated by twice its first term, which gives \eqref{eq:sgeneral}. In particular $s<\infty$, so by Lemma \ref{lem:Snestimate}, $||S_n - S_{n-1}|| = ||S_{n-1}\mathcal{E}_{n-1}|| \leq e^s||\mathcal{E}_{n-1}||$ is summable and $S_n \to S_\infty$. Finally, $\det (I + \mathcal{E}_j) = \det (Y_j^{-1}v_jY_j) = 1$ for every $j$, so $\det S_n = 1$ for all $n$ and $\det S_\infty = 1$. This proves \eqref{eq:formula}, and $\det D_c\Psi = \lambda_\infty^2 \det R_\infty$ gives the statement about signs. Continuity of $\det D_c\Psi$ follows from $\Psi$ being $C^1$.
\end{proof}

\section{The map $\Psi$ near infinity}

Here, we will prove Theorem \ref{thm:annular}.

\begin{proof}[Proof of Theorem \ref{thm:annular}]
We take
\begin{equation}
\label{eq:R}
R \geq \max \left \{ \kappa_0, \; (1+\pi) C_* \right \},
\end{equation}
where $\kappa_0$ is as in \eqref{eq:kappa0} and $C_*$ is the constant produced at the end of the proof.

By Corollary \ref{lem:Pescape} with $m=0$, $\kappa = \kappa_0$ and $S = |c| \geq R \geq \kappa_0 > 2$, we have
\begin{equation}
\label{eq:thma1}
|P_n(c)| \geq |c|(|c|-1)^{2^n-1} \geq (|c|-1)^{2^n} \geq (R-1)^{2^n}.
\end{equation}
In particular the escape time \eqref{eq:escape} is $m(c) = 0$, and $\kappa := \inf_{i\geq 0} |P_i(c)| = |c| \geq R > 1$, so the last part of Lemma \ref{lem:Rnclose} implies that
\begin{equation}
\label{eq:thma2}
||R_{\infty} - I|| \leq \frac{1}{2|c|-1}, \quad \mathcal{H} (R_{\infty}) \leq \frac{|c|}{|c|-1}.
\end{equation}

By Proposition \ref{prop:derivative}, $\Psi$ is $C^1$ on $\{|c|>R\}$ and $D_c\Psi = \lambda_\infty S_\infty R_\infty$ with $\lambda_\infty >0$ and $\det S_\infty = 1$. Since $\lambda_\infty$ is a positive scalar, the linear dilatations satisfy
\begin{equation}
\label{eq:thma3}
\mathcal{H}(D_c \Psi) \leq \mathcal{H}(S_{\infty}) \mathcal{H}(R_{\infty}),
\end{equation}
and, moreover,
\[ \det D_c\Psi = \lambda_{\infty}^2 \det S_{\infty} \det R_{\infty} > 0.\]
As $m(c)=0$ and $|P_0(c)| = |c|$, the bound \eqref{eq:sgeneral} reads
\begin{equation}
\label{eq:thma4}
s \leq \frac{4C'K^2|c|}{(|c|-1)^2} = O_{K,\theta}(1/|c|);
\end{equation}
we note that the ratio-test condition $(|c|-1)^2 \geq 2K$ used in \eqref{eq:sgeneral} is exactly the second requirement in \eqref{eq:kappa0}. Combining Lemma \ref{lem:Snestimate} with \eqref{eq:thma2}, \eqref{eq:thma3} and \eqref{eq:thma4}, we have
\[ \mathcal{H}( D_c \Psi) \leq \frac{ e^{2s} |c|}{|c|-1} = 1+O_{K,\theta}(1/|c|), \quad \det D_c\Psi >0.\]
Moreover, $||S_{\infty} - I|| \leq e^s - 1 = O_{K,\theta}(1/|c|)$ by \eqref{eq:Sn1}, $||R_{\infty} - I|| \leq (2|c|-1)^{-1}$, and, by \eqref{eq:lambda} and \eqref{eq:thma1},
\[ |\log \lambda_\infty| \leq \sum_{j=0}^{\infty} \frac{CK|c|}{|P_j(c)|^2} \leq CK|c| \sum_{j=0}^{\infty} (|c|-1)^{-2^{j+1}} \leq \frac{2CK|c|}{(|c|-1)^2} = O_{K,\theta}(1/|c|),\]
so that $\lambda_\infty = 1 + O_{K,\theta}(1/|c|)$. Hence
\[ D_c\Psi = I + E(c), \quad ||E(c)|| \leq \frac{C_*(K,\theta) }{|c|}\]
for some constant $C_*$ depending only on $K$ and $\theta$. In particular $D_c\Psi \to I$ as $c\to \infty$, and $\Psi$ is a $C^1$ local homeomorphism on $\{|c| > R\}$ with the dilatation bound above.

For global injectivity, let $c_1,c_2 \in \{ |c| \geq R \}$ with $c_1 \neq c_2$, and put $\varrho = \max \{ |c_1|,|c_2| \}$; say $|c_2| = \varrho$. 
By joining $c_1$ to $c_2$ via a path $\gamma$ that consists of radial line segments and an arc of the circle $\{|c| = \varrho \}$, we can bound the length of $\gamma$ by $(1+\pi)|c_1-c_2|$.
Therefore
\[ |\Psi(c_1) - \Psi (c_2) - (c_1-c_2)| = \left | \int_{\gamma} E(c) \, dc \right | \leq \sup_{\gamma} ||E|| \cdot (1+\pi) |c_1-c_2| \leq \frac{ (1+\pi) C_*}{R} |c_1-c_2| < |c_1-c_2| \]
once $R > (1+\pi) C_*$, which is the second requirement in \eqref{eq:R}. Hence $\Psi(c_1) \neq \Psi(c_2)$, and we conclude that $\Psi$ is injective on $\{|c|>R\}$. Being an injective $C^1$ map with $\det D_c\Psi >0$ and bounded linear dilatation, $\Psi$ is a quasiconformal homeomorphism of $\{|c|>R\}$ onto its image. This completes the proof.
\end{proof}

\section{The obstruction - folds}
\label{sec:fold}

In this section, we will discuss why the map $\Psi$ cannot be continued via the equation $f_0^n(\Psi(c)) = \psi_c(P_n(c))$. Recall first from Lemma \ref{lem:alg1} that
\[ \sign \det D_cP_n = \sign \det R_n, \quad R_n = I + \sum_{j=1}^n \Pi_j^{-1}.\]
Orientation behaviour is produced entirely by the $+I$ term here, that is, by the direct dependence of $f_c$ on $c$. There is no issue when $K=1$, as then $R_n$ is complex-linear, so $\det R_n = |R_n|^2 \geq 0$ and orientation reversal is not possible. 

Our key lemma in this section is the following. It provides a finite checkable condition on whether a parameter reverses orientation.

\begin{lemma}[Frozen Sign Lemma]
\label{lem:frozen}
Let $K\geq 1$, $\theta \in (-\pi/2,\pi/2]$, $c\in \C$ and write $D_n:= D_cP_n(c)$.
Suppose that for some $m\geq 0$ we have $\ell(D_m) \geq 1$ and $|P_m(c)| \geq \max \{ |c|,2 \}$. Then $\det D_n \neq 0$ and
\[ \sign \det D_n = \sign \det D_m\]
for every $n\geq m$.
\end{lemma}

\begin{proof}
We have
\begin{equation}
\label{eq:frozen1}
D_{n+1} = (A_n h)D_n + I,\quad D_0 =I,
\end{equation}
where, as in the previous section, $A_n$ is multiplication by $2h(P_n(c))$.

First, if $|P_n(c)| \geq \max \{|c|,2\}$, then by \eqref{eq:hsv}
\[ |P_{n+1}(c)| \geq |h(P_n(c))|^2 - |c| \geq |P_n(c)|^2 - |c| \geq |P_n(c)|^2 - |P_n(c)| \geq |P_n(c)|,\]
using $|P_n(c)| \geq 2$ in the last step, so $|P_{n+1}(c)| \geq |P_n(c)| \geq \max\{|c|,2\}$. Hence this hypothesis propagates for all $n\geq m$.

Next, we claim that if $\ell(D_n) \geq 1$ and $|P_n(c)| \geq \max \{|c|,2\}$, then $\ell (D_{n+1}) \geq 3$. Writing $X_n = A_nhD_n$, we have by \eqref{eq:hsv} and supermultiplicativity of $\ell$,
\[ \ell(X_n) \geq \ell(A_n)\ell(h)\ell(D_n) = 2|h(P_n(c))| \cdot 1 \cdot \ell(D_n) \geq 2|P_n(c)| \geq 4,\]
and therefore, by \eqref{eq:frozen1} and Weyl's inequality,
\[ \ell(D_{n+1}) \geq \ell(X_n) - ||I|| \geq 4-1=3.\]
In particular, both hypotheses propagate, so $\ell(D_n)\geq 1 >0$ for all $n\geq m$, whence $\det D_n \neq 0$ and the sign is well-defined.

Finally, $||X_n^{-1}|| = \ell (X_n)^{-1} \leq \tfrac14$ by the above. Then
\[ D_{n+1} = X_n + I = X_n(I+E_n), \quad E_n := X_n^{-1},\]
where $||E_n|| \leq \tfrac14$. For $t\in [0,1]$,
\[ \ell ( I +tE_n) \geq 1 - t||E_n|| \geq \tfrac34 >0,\]
so $\det (I+tE_n) \neq 0$ along the path, and thus $\det(I+E_n)>0$ as $\det I = 1$. As $\det (A_nh) = 4|h(P_n(c))|^2K >0$, we have
\[ \det D_{n+1} = \det (A_nh) \cdot \det D_n \cdot \det (I+E_n),\]
where the first and third terms on the right are strictly positive. It follows that $\sign \det D_{n+1} = \sign \det D_n$, as required.
\end{proof}

Recall from Lemma \ref{lem:Rnclose} that $\sigma(c)<1$ implies $\det D_cP_n >0$ for all $n$. As an immediate consequence of this, we note the following: orientation reversal forces a quantitatively strong close return of the critical orbit to the critical point.

\begin{corollary}
\label{cor:fold}
Let $c\in \C\setminus \mathcal{M}_{K,\theta}$ and suppose either that $\det D_cP_n(c) <0$ for some $n$, or that $c$ is an orientation reversing parameter. Then $\sigma(c) \geq 1$. In particular $|h(P_i(c))| < 1$, and hence $|P_i(c)|<1$, for some $i\geq 0$; so for every $r>1$, neither $\Omega_r$ nor $\Omega^*$ contains a parameter with $\det D_cP_n <0$ for some $n$, or an orientation reversing parameter.
\end{corollary}

\begin{proof}
If $\sigma(c)<1$, then $\det D_cP_n(c) >0$ for all $n$ and $\det R_\infty >0$ by Lemma \ref{lem:Rnclose}, and then $\det D_c\Psi >0$ by Proposition \ref{prop:derivative}. This gives the first assertion. If $|h(P_i(c))| \geq 1$ for all $i\geq 0$, then $\sigma(c) \leq \sum_{j\geq 1}2^{-j} = 1$, with equality only if $|h(P_i(c))|=1$ for every $i$, which is impossible since $P_i(c) \to \infty$; so $\sigma(c)<1$. Hence $\sigma(c) \geq 1$ forces $|h(P_i(c))|<1$ for some $i$, and $|P_i(c)| \leq |h(P_i(c))|$ by \eqref{eq:hsv}.
\end{proof}

Of course, we haven't yet shown that orientation reversing parameters exist. Let us address this now - see Figure \ref{fig:K5}.

\begin{figure}
    \centering
    \includegraphics[width=3in]{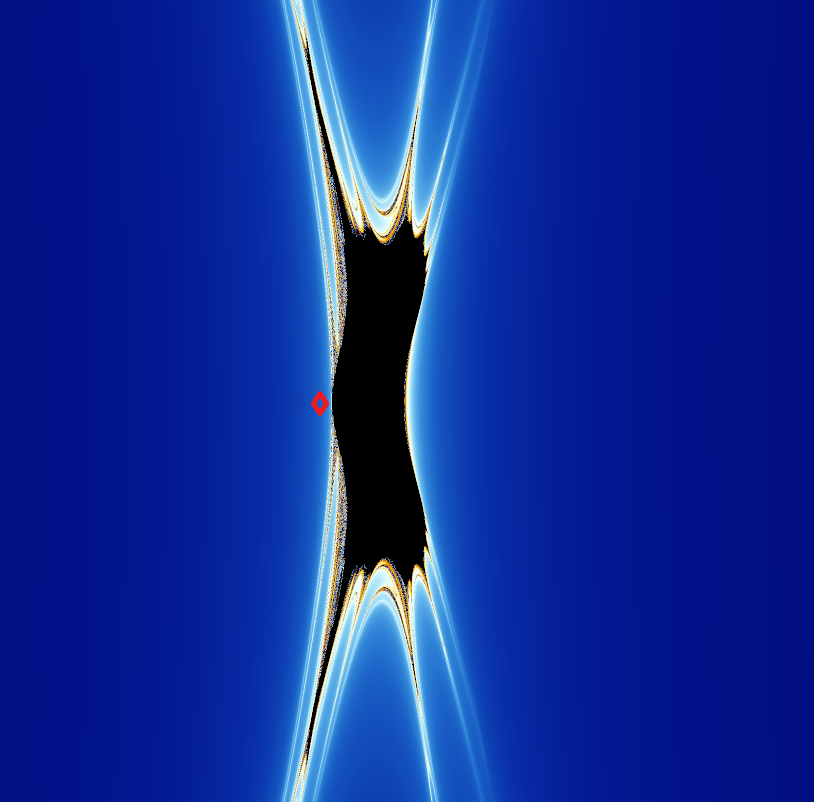}
    \caption{$\mathcal{M}_{5,0}$ with the location of the parameter $c=-1/10$ marked}
    \label{fig:K5}
\end{figure}

\begin{proof}[Proof of Theorem \ref{thm:fold}]
Consider $K=5$, $\theta = 0$ and $c = -1/10$, an example studied in \cite[Section 7]{BF25}. By \cite[Theorem 6.4]{FG10}, $\mathcal{M}_{5,0} \cap \R = [-2/25, 1/100]$ and so $c\notin \mathcal{M}_{5,0}$.

As $c$ and $K$ are real and $\theta = 0$, the orbit $P_n(c)$ is real and $D_cP_n(c)$ is diagonal. Writing
\[ D_cP_n(c) = \begin{pmatrix} p_n & 0\\0 & q_n \end{pmatrix},\]
\eqref{eq:DcPrec} gives $p_0=q_0=1$ and
\[ p_{n+1} = 2K^2P_n(c)p_n+1, \quad q_{n+1} = 2KP_n(c)q_n+1.\]
Then $q_1=0$, $p_1=-4$, $P_1(c) = 3/20$, $p_2 = -29$ and $q_2=1$. Moreover $P_{n+1}(c) = 25P_n(c)^2 - 1/10 > P_n(c)$ whenever $P_n(c) \geq 3/20$, so $(P_n(c))_{n\geq 1}$ is increasing and $P_n(c) \geq 3/20$ for all $n\geq 1$. Thereafter, $2K^2P_n(c) \geq 7.5$ and $2KP_n(c) \geq 1.5$, so that if $p_n \leq -29$ and $q_n \geq 1$ then
\[ p_{n+1} \leq 7.5 \cdot (-29) + 1 < -29, \quad q_{n+1} \geq 1.5\cdot 1 + 1 > 1.\]
By induction, $p_n \leq -29$ and $q_n \geq 1$ for all $n\geq 2$, and hence $\det D_cP_n(c) = p_nq_n < 0$ for all $n\geq 2$.


We now show that $\det R_\infty <0$. Note that $R_n \to R_\infty$ and $\det R_n <0$ for $n\geq 2$ only give $\det R_\infty \leq 0$, so an argument is needed. In the present real situation, $\Pi_n = \operatorname{diag} ( \prod_{i<n} 2K^2P_i(c), \prod_{i<n} 2KP_i(c) )$, so by \eqref{eq:DcPn}, $R_n = \operatorname{diag}(r^{(1)}_n, r^{(2)}_n)$ with
\[ r^{(1)}_n = 1 + \sum_{j=1}^n \prod_{i=0}^{j-1} \frac{1}{2K^2P_i(c)}, \qquad r^{(2)}_n = 1 + \sum_{j=1}^n \prod_{i=0}^{j-1} \frac{1}{2KP_i(c)}.\]
Every product here is negative, since $P_0(c) = -1/10 <0$ and $P_i(c) >0$ for $i\geq 1$, so both sequences are decreasing in $n$. For $r_n^{(2)}$, the first two terms of the sum are $(2KP_0)^{-1} = -1$ and $(2KP_0)^{-1}(2KP_1)^{-1} = -2/3$, so $r^{(2)}_\infty \leq r^{(2)}_2 = -2/3 <0$. For $r_n^{(1)}$, we have $(2K^2P_0)^{-1} = -1/5$ and $2K^2P_i(c) \geq 15/2$ for $i\geq 1$, so the $j$-th product has modulus at most $\tfrac15 (2/15)^{j-1}$, and $r^{(1)}_\infty \geq 1 - \tfrac15 (1 - 2/15)^{-1} = 10/13 >0$. Hence $\det R_\infty = r^{(1)}_\infty r^{(2)}_\infty <0$. In fact, numerically, $r^{(1)}_\infty \approx 0.772$ and $r^{(2)}_\infty \approx -0.814$. Then by Proposition \ref{prop:derivative}, $\det D_c\Psi(-1/10) <0$, so $c=-1/10$ is an orientation reversing parameter.

Now let $U$ be any connected open subset of $\C\setminus \mathcal{M}_{5,0}$ containing $-1/10$ and a neighbourhood of infinity; such sets exist, for instance $\{|c|>R\}$ together with a small neighbourhood of the segment $[-R,-1/10]$, which is a compact subset of $\C \setminus \mathcal{M}_{5,0}$ because $\mathcal{M}_{5,0} \cap \R = [-2/25,1/100]$. By Theorem \ref{thm:annular}, $\det D_c\Psi >0$ on $\{|c|>R\}$, so $\det D_c \Psi$ takes both signs on $U$ and $\Psi$ has a fold in $U$. As explained in the introduction, $\Psi$ is therefore not injective on $U$, and so not injective on $\C\setminus\mathcal{M}_{5,0}$. In particular it is not a quasiconformal map of $\C\setminus\mathcal{M}_{5,0}$. Finally, since $\det D_c\Psi$ is continuous on $(-\infty,-2/25) \subset \C\setminus\mathcal{M}_{5,0}$ by Proposition \ref{prop:derivative}, it vanishes at some point of $(-R,-1/10)$, so the fold locus meets the negative real axis.
\end{proof}

One could ask whether the fact that $K$ is relatively large plays into the existence of folds. In quasiregular dynamics, the condition that the degree is larger than the maximal dilatation, in this case $K<2$, sometimes yields better behaviour. However, numerical investigation reveals orientation reversing parameters for every $K$ tested in the range $[1.015,5]$, with $\theta=0$. For example, if $K=1.4$ and $\theta = 0$, then $c = -0.399492+0.401356 i$ satisfies the hypotheses of Lemma \ref{lem:frozen} at $m=148$ with $\det D_cP_m<0$, see Figure \ref{fig:K1point4}, and similarly for $K=1.015$ and $c=-0.722073-0.143868i$ at $m=3439$, see Figure \ref{fig:K1point015}. 

\begin{figure}
\begin{subfigure}[h]{0.4\linewidth}
\includegraphics[width=\linewidth]{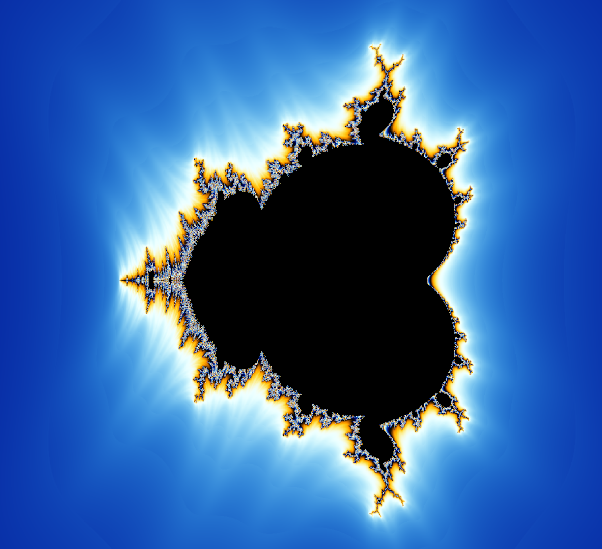}
\end{subfigure}
\hfill
\begin{subfigure}[h]{0.5\linewidth}
\includegraphics[width=\linewidth]{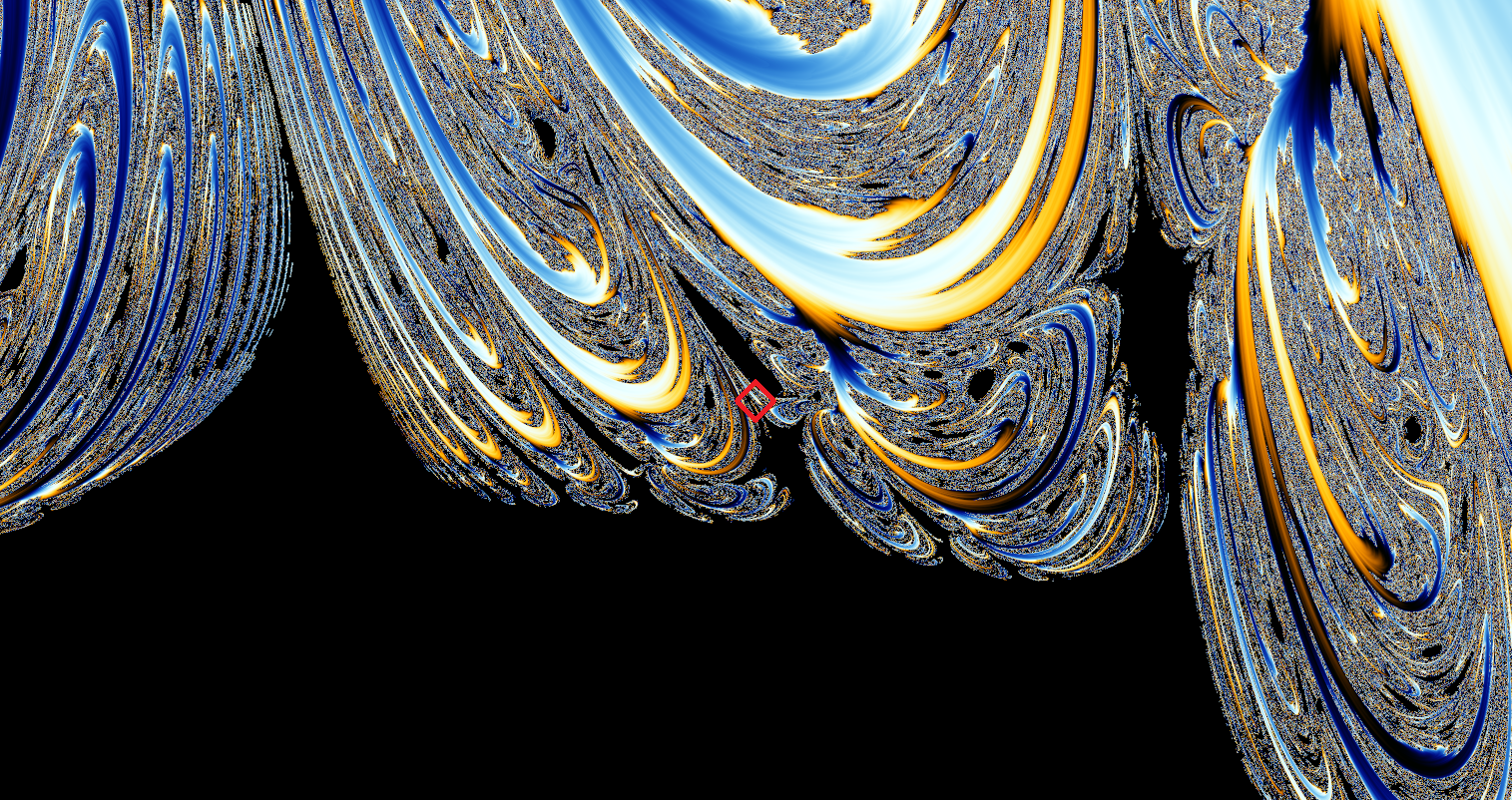}
\end{subfigure}%
\caption{On the left, $\mathcal{M}_{1.4,0}$ and on the right, a zoom near $c=-0.399492+0.401356 i$, marked.}
\label{fig:K1point4}
\end{figure}

\begin{figure}
\begin{subfigure}[h]{0.4\linewidth}
\includegraphics[width=\linewidth]{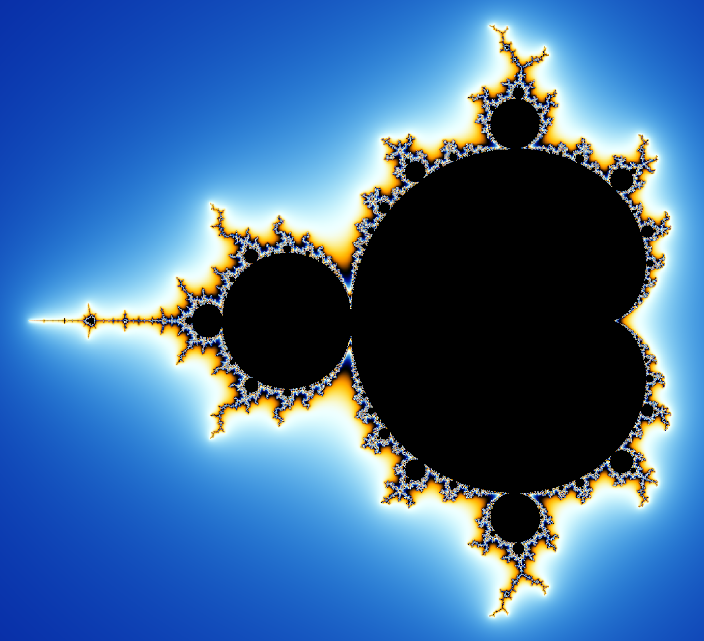}
\end{subfigure}
\hfill
\begin{subfigure}[h]{0.5\linewidth}
\includegraphics[width=\linewidth]{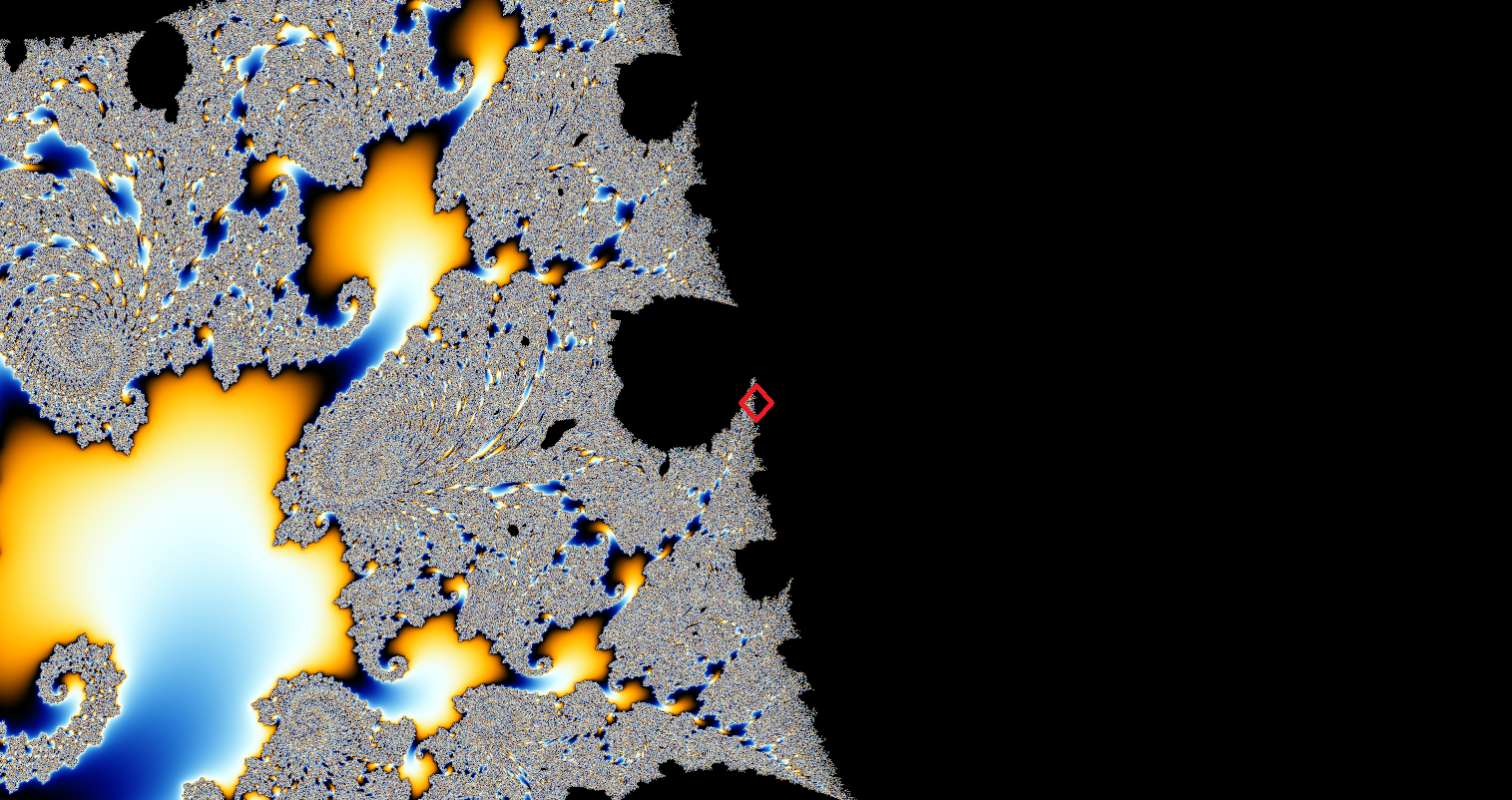}
\end{subfigure}%
\caption{On the left, $\mathcal{M}_{1.015,0}$ and on the right, a zoom near $c=-0.722073-0.143868i$, marked.}
\label{fig:K1point015}
\end{figure}


\section{Green's formulation of the uniformizing map}
\label{sec:green}

In this section we prove Theorem \ref{thm:potential}. Before giving the proof, we comment on the role of the parameter $r$.
If $1<r_1<r_2$ then $\Omega_{r_2} \subset \Omega_{r_1}$, and hence
$t_*(r_1) \leq t_*(r_2)$; so decreasing $r$ towards $1$ allows us to
take $t$ smaller, at the cost of the dilatation bound $r/(r-1)$
deteriorating. This trade-off is unavoidable: by Theorem \ref{thm:fold}
there is an orientation reversing parameter $c_0$ for $K=5$, $\theta=0$,
and by Corollary \ref{cor:fold} we have $\inf_i |P_i(c_0)| \leq 1$, so
$c_0 \notin \Omega_r$ for every $r>1$. As $c_0 \notin \mathcal{M}_{5,0}$
we have $G_{\mathcal{M}}(c_0) >0$, and therefore
\[ t_*(r) \geq G_{\mathcal{M}}(c_0) \quad \text{for every } r>1,
\qquad \text{so} \qquad t_0 \geq G_{\mathcal{M}}(c_0) >0.\]
Numerically, $G_{\mathcal{M}}(-1/10) \approx 0.5844$ when $K=5$ and
$\theta = 0$. In particular, Theorem \ref{thm:potential} cannot be
extended to all of $\C \setminus \mathcal{M}_{K,\theta}$, which is
consistent with Theorem \ref{thm:fold}.

The following lemma shows that $t_0>0$ for every $K\geq 1$ when $\theta = 0$, and that $t_0$ stays bounded away from $0$ as $K\to 1$, even though for $K=1$ the map $\Psi$ is a conformal isomorphism of the whole exterior of $\mathcal{M}$. It also makes precise the remark, made after Lemma \ref{lem:Rnclose}, that the condition $\sigma<1$ is sufficient but not necessary for orientation preservation.

\begin{lemma}
\label{lem:tzero}
Let $\theta = 0$, $K\geq 1$ (with the convention that $h = \mathrm{id}$ when $K=1$) and $c_K := 1/(2K)$. Then $c_K \notin \mathcal{M}_{K,0}$ and $\sigma(c_K) \geq 1$, so that $t_0 \geq G_{\mathcal{M}}(c_K) >0$. Moreover, $G_{\mathcal{M}}(c_K) \geq 1/26$ for $1 \leq K \leq 2$.
\end{lemma}

\begin{proof}
Everything is real, and $h(x) = Kx$ for real $x$. Set $y_n := h(P_n(c_K)) = KP_n(c_K)$. Then $y_0 = 1/2$ and $y_{n+1} = K(y_n^2 + c_K) = Ky_n^2 + \tfrac12 \geq y_n^2 + \tfrac12$. Let $z_0 = 1/2$ and $z_{n+1} = z_n^2 + \tfrac12$; then $z_{n+1} - z_n = (z_n - \tfrac12)^2 + \tfrac14 \geq \tfrac14$, so $z_n \to \infty$, and a direct computation gives $z_1 = 3/4$, $z_2 = 17/16$, $z_3 \approx 1.6289$, $z_4 \approx 3.1533$ and $z_5 \approx 10.4435$. By induction $y_n \geq z_n$, so $P_n(c_K) = y_n/K \to \infty$ and $c_K\notin \mathcal{M}_{K,0}$. The first term of $\sigma(c_K)$ is $(2|h(c_K)|)^{-1} = (2y_0)^{-1} = 1$, so $\sigma(c_K) \geq 1$ and $c_K \notin \Omega^*$. Hence $\{G_{\mathcal{M}} > s\} \not\subset \Omega^*$ for every $s < G_{\mathcal{M}}(c_K)$, which gives $t_0 \geq G_{\mathcal{M}}(c_K)$, and $G_{\mathcal{M}}(c_K)>0$ as $c_K \notin \mathcal{M}_{K,0}$.

If $1\leq K \leq 2$, then $P_5(c_K) = y_5/K \geq z_5/2 \geq 5.22 \geq \max \{ |c_K|, 5\}$, so Lemma \ref{lem:Glower} with $\kappa = 5$ gives $G_{c_K}(P_5(c_K)) \geq \log 5.22 - 2/5.22 > 1.26$, and therefore $G_{\mathcal{M}}(c_K) = 2^{-5} G_{c_K}(P_5(c_K)) > 1.26/32 > 1/26$.
\end{proof}

\begin{proof}[Proof of Theorem \ref{thm:potential}]
Recall $\kappa_0$ from \eqref{eq:kappa0}, $\kappa_c = \max\{|c|,\kappa_0\}$ and the escape time $m(c)$ from \eqref{eq:escape}. If $|c| \geq \kappa_0$ then $m(c) = 0$. If $|c| < \kappa_0$, then $\kappa_c = \kappa_0$, and from $G_c(P_i(c)) = 2^iG_{\mathcal{M}}(c)$ and Corollary \ref{cor:bott} we see that $|P_i(c)| \geq \kappa_0 = \kappa_c$ as soon as $2^iG_{\mathcal{M}}(c) > \log \kappa_0 + \log (K^2+1) =: A$. Hence on $\{G_{\mathcal{M}}>t\}$,
\[ m(c) \leq M(t) := \max \left \{ 1, \; \left \lceil \log_2 \frac{A}{t} \right \rceil +1 \right \} \leq 2 + \log_2^+ \frac{A}{t}.\]
In particular, the escape time is at most logarithmic in $1/t$.

We first show that $\sigma_t<1$. Fix $t'\in (t_0,t)$. Then $\{G_{\mathcal{M}} \geq t\} \subset \{G_{\mathcal{M}}>t'\} \subset \Omega^*$, so $\sigma<1$ on $\{G_{\mathcal{M}}\geq t\}$. The function $\sigma$ is continuous on $\C\setminus \mathcal{M}_{K,\theta}$: given $c \notin \mathcal{M}_{K,\theta}$, \eqref{eq:thmp1} gives $|P_{m(c)+1}(c)| \geq 3\kappa_c > \kappa_c$, and this strict inequality persists for $c'$ in a neighbourhood of $c$, so Corollary \ref{lem:Pescape} applied at time $m(c)+1$ gives $|P_n(c')| \geq 3^{2^{n-m(c)-1}}$ for all $n\geq m(c)+1$ and all $c'$ in that neighbourhood; hence the series defining $\sigma$ converges uniformly there, and its terms are continuous since $P_i(c') \neq 0$. Put $\varrho := \max\{\kappa_0, e^{t+C_1}\}$, with $C_1$ from \eqref{eq:Mlower}. If $|c| \geq \varrho$, then Corollary \ref{lem:Pescape} with $m=0$ and $\kappa = \kappa_0$ shows that $(|P_i(c)|)_{i\geq 0}$ is non-decreasing, so $\inf_i|P_i(c)| = |c| \geq \kappa_0 \geq 4$ and $\sigma(c) \leq (2|c|-1)^{-1} \leq 1/7$ by Lemma \ref{lem:Rnclose}. The remaining set $\{G_{\mathcal{M}} \geq t\} \cap \{|c| \leq \varrho\}$ is closed by Lemma \ref{lem:Gcont}, hence compact, and it is contained in $\C\setminus \mathcal{M}_{K,\theta}$ as $G_{\mathcal{M}}$ vanishes on $\mathcal{M}_{K,\theta}$, so $\sigma$ attains a maximum less than $1$ on it. Hence $\sigma_t = \sup_{\{G_{\mathcal{M}}>t\}}\sigma <1$, as claimed, and if $t>t_*(r)$ then $\{G_{\mathcal{M}}>t\} \subset \Omega_r$ gives $\sigma_t \leq (2r-1)^{-1}$ by Lemma \ref{lem:Rnclose}.

By Proposition \ref{prop:derivative}, $\Psi$ is $C^1$ on $\C \setminus \mathcal{M}_{K,\theta}$ and $D_c\Psi = \lambda_\infty S_\infty R_\infty$ with $\lambda_\infty>0$ and $\det S_\infty = 1$. On $\{G_{\mathcal{M}}>t\}$ we have $\sigma(c) \leq \sigma_t<1$, so Lemma \ref{lem:Rnclose} gives
\begin{equation}
\label{eq:thmp2}
||R_n- I|| \leq \sigma_t, \quad \det R_\infty > 0, \quad \mathcal{H}(R_{\infty}) \leq \frac{1+\sigma_t}{1-\sigma_t}.
\end{equation}
We note in particular that $\{G_{\mathcal{M}}>t\}$ contains no orientation reversing parameters, by Corollary \ref{cor:fold}, which is why $t>t_0$ is the correct hypothesis for Theorem \ref{thm:potential}.

We now estimate $s$ on $\{G_{\mathcal{M}}>t\}$ using \eqref{eq:sgeneral}, with $m = m(c) \leq M(t)$ and $S = |P_m(c)|$. If $|c| \geq \kappa_0$ then $m=0$, $S = |c| \geq 4$, and $|c|/(S-1)^2 \leq 4/9$; if $|c|<\kappa_0$ then $S-1 \geq \kappa_0-1 \geq 3$ and $|c|/(S-1)^2 \leq \kappa_0/9$. Hence
\[ s \leq \left ( \pi M(t) + C''(K,\theta) \right ) K^{M(t)}, \qquad C'' := \tfrac49 C'K^2\kappa_0 .\]
Since $M(t) \leq 2 + \log_2^+(A/t)$, we have $K^{M(t)} \leq K^2 \max\{1, (A/t)^{\log_2 K}\}$, and as $A \geq \log 4 + \log 2 > 1$ this gives the bound
\[ s(t) \leq C(K,\theta) \left ( 1 + \log^+ \tfrac{1}{t} \right ) \max \left \{ 1, t^{-\log_2 K} \right \} \]
stated in the theorem. By Lemma \ref{lem:Snestimate}, whose conclusion holds for all $s$ with no smallness assumption, $\mathcal{H}(S_{\infty}) \leq e^{2s}$. Combining with \eqref{eq:thmp2}, we obtain
\[ \mathcal{H}(D_c\Psi) \leq \frac{1+\sigma_t}{1-\sigma_t}\,e^{2s}, \quad \det D_c \Psi = \lambda_\infty^2 \det R_\infty > 0.\]

We next show that $\Psi$ is a homeomorphism from $\{G_{\mathcal{M}} > t \}$ onto $\{G_0>t\}$. By the above, $\Psi$ is $C^1$ with $\det D_c \Psi >0$ on $\{G_{\mathcal{M}}>t\}$, so it is a local homeomorphism there, and $\mathcal{H}(D_c\Psi)$ is bounded by a function of $t$. Next, $G_c = G_0 \circ \psi_c$ on a neighbourhood of infinity by \cite[Definition 3.6]{BF25}, so for $n$ large,
\[ G_{\mathcal{M}}(c) = G_c(c) = 2^{-n} G_c(P_n(c)) = 2^{-n} G_0 ( \psi_c(P_n(c)) ) = 2^{-n} G_0( f_0^n(\Psi(c))) = G_0(\Psi(c)),\]
using the functional equations of $G_c$, of the extended $\psi_c$, and of $G_0$. Hence $\Psi ( \{G_\mathcal{M} > t \}) \subset \{G_0 > t \}$.

For properness, suppose $\Psi (c_n) \to w$ with $G_0(w) >t$. Then $G_{\mathcal{M}}(c_n) = G_0(\Psi(c_n)) \to G_0(w)$, so eventually $G_{\mathcal{M}}(c_n) > t+\delta$ for some $\delta>0$; and $G_{\mathcal{M}}(c_n) \leq \log \max\{|c_n|,1\} + \log(K^2+1)$ together with \eqref{eq:Mlower} shows that $(c_n)$ is bounded. Any limit point $c$ satisfies $G_{\mathcal{M}}(c) \geq t+\delta$ by Lemma \ref{lem:Gcont}, hence lies in the open set $\{G_{\mathcal{M}}>t\}$, and $\Psi(c) = w$ by continuity of $\Psi$. We conclude that $\Psi$ is proper, and a proper local homeomorphism is a covering map onto its image; since $\{G_0>t\}$ is connected by Lemma \ref{lem:jordan}, the image is all of $\{G_0>t\}$.

We claim $\{ G_{\mathcal{M}} > t \}$ is connected. The restriction of $\Psi$ to any connected component of $\{G_{\mathcal{M}}>t\}$ is again a proper local homeomorphism, hence a covering map onto the connected set $\{G_0>t\}$, on which $G_0$ is unbounded. Thus $G_{\mathcal{M}}$, and therefore $|c|$ by Corollary \ref{cor:bott}, is unbounded on each component. Set $\varrho' := \max\{\varrho, R\}$, with $R$ from Theorem \ref{thm:annular}. Then $\{|c| > \varrho'\} \subset \{G_{\mathcal{M}}>t\}$ by \eqref{eq:Mlower}, and this set is connected, so it lies in a single component, which every component therefore meets. So there is only one component.

Finally, a covering map over a connected base has constant degree. Let $w\in \{G_0>t\}$ with $G_0(w) > \log \varrho' + \log(K^2+1)$. If $\Psi(c) = w$ then $G_{\mathcal{M}}(c) = G_0(w)$, so $|c|>\varrho' \geq R$ by Corollary \ref{cor:bott}, and $\Psi$ is injective on $\{|c|>R\}$ by Theorem \ref{thm:annular}. Thus $w$ has exactly one preimage, the degree is $1$, and $\Psi$ is a homeomorphism onto $\{G_0>t\}$. Being a $C^1$ homeomorphism with bounded linear dilatation, $\Psi$ is quasiconformal. This completes the proof.
\end{proof}

\begin{proof}[Proof of Corollary \ref{cor:potential}]
Let $t>t_0$ and pick $t' \in (t_0,t)$. By Theorem \ref{thm:potential} applied at $t'$, $\Psi$ is a homeomorphism from $\{ G_{\mathcal{M}} >t' \}$ onto $\{G_0 >t' \}$, and $G_0 \circ \Psi = G_{\mathcal{M}}$. Hence $\gamma := \{G_{\mathcal{M}} = t\} = \Psi^{-1}(e^t\partial I(f_0))$ is a Jordan curve by Lemma \ref{lem:jordan}. 

Now $\{G_{\mathcal{M}} > t\} \cup \{\infty\}$ is a connected subset of $\widehat{\C} \setminus \gamma$: it is connected by Theorem \ref{thm:potential}, and it contains a neighbourhood of $\infty$ by \eqref{eq:Mlower}. It is open, and it is also closed in $\widehat{\C}\setminus \gamma$, because its closure in $\widehat{\C}$ is contained in $\{G_{\mathcal{M}} \geq t\} \cup \{\infty\} = \{G_{\mathcal{M}}>t\} \cup \gamma \cup \{\infty\}$ by Lemma \ref{lem:Gcont}. Hence it is the unbounded complementary component of $\gamma$, and $\{G_{\mathcal{M}} \leq t\} = \C \setminus \{G_{\mathcal{M}}>t\}$ is the closure of the bounded complementary component, which is a closed topological disk by the Schoenflies theorem. It contains $\mathcal{M}_{K,\theta}$, since $G_{\mathcal{M}}$ vanishes there. The external rays and equipotentials at potentials in $(t',\infty)$ are the preimages under $\Psi|_{\{G_{\mathcal{M}}>t'\}}$ of the radial arcs $\{ re^{i\phi} : r > e^{t'}|b_{K,\theta}(\phi)| \}$ and of the curves \eqref{eq:equipot}; as $t'$ may be taken arbitrarily close to $t_0$, this gives rays and equipotentials at all potentials greater than $t_0$.
\end{proof}

\end{document}